\documentclass[11pt,oneside]{amsart}
\usepackage{amsfonts}
\usepackage{amsmath}
\usepackage{amssymb}
\usepackage{mathrsfs}
\usepackage{geometry}
\usepackage{graphicx}
\usepackage{subfigure}
\usepackage{cite}
\usepackage{hyperref}
\usepackage{microtype}
\usepackage{enumerate}
\usepackage{mathrsfs}
\usepackage{tikz}%% drawing
\usepackage{tikz-cd}%% drawing
\usepackage{color}
\usepackage{ragged2e}
 \usepackage{float}
\allowdisplaybreaks[1]

\numberwithin{equation}{section}

\newcommand*{\circled}[1]{\lower.7ex\hbox{\tikz\draw (0pt, 0pt)%
    circle (.5em) node {\makebox[1em][c]{\small #1}};}}

\newcommand{\la}{\lambda}
\newcommand{\al}{\alpha}

\newcommand{\ga}{\gamma}
\newcommand{\Ga}{\Gamma}

\newcommand{\ve}{\varepsilon}

\newcommand{\R}{\mathbb{R}}

\newcommand{\C}{\mathbb{C}}

\newcommand{\Z}{\mathbb{Z}}
\newcommand{\T}{\mathbb{T}}

\newcommand{\ccc}{\cdot\cdot\cdot}

\newcommand{\n}[1]{\Vert #1\Vert }

\newcommand{\bbn}[1]{\Big\Vert #1 \Big \Vert }
\newcommand{\lr}[1]{\left\{ #1\right\} }
\newcommand{\lrc}[1]{\left[ #1\right] }
\newcommand{\lrs}[1]{\left( #1\right) }

\newcommand{\lra}[1]{\langle #1\rangle}

\newcommand{\babs}[1]{\big | #1 \big| }
\newcommand{\bbabs}[1]{\Big | #1 \Big| }
\newcommand{\wt}[1]{\widetilde{#1} }
\newcommand{\wq}{\infty}
\newcommand{\pa}{\partial}

\newcommand{\ol}{\overline}

\begin{document}
\newtheorem{theorem}{Theorem}[section]
\newtheorem{lemma}[theorem]{Lemma}
\theoremstyle{definition}
\newtheorem{definition}[theorem]{Definition}
\newtheorem{example}[theorem]{Example}
\newtheorem{remark}[theorem]{Remark}

\numberwithin{equation}{section}

\newtheorem{proposition}[theorem]{Proposition}
\newtheorem{corollary}[theorem]{Corollary}
\newtheorem{goal}[theorem]{Goal}
\newtheorem{algorithm}{Algorithm}

\renewcommand{\figurename}{Fig.}

\title[Uniqueness for the Quadratic NLS]{Unconditional Uniqueness for the Energy-critical and Energy-supercritical Quadratic NLS}

\author[S. Shen]{Shunlin Shen}
\address{School of Mathematical Sciences, University of Science and Technology of China, Hefei, 230026, China}
\email{slshen@ustc.edu.cn}


\dedicatory{}

%    Abstract is required.
\begin{abstract}
We study the quadratic nonlinear Schr\"{o}dinger equation in energy-critical and energy-supercritical regimes and establish unconditional uniqueness at critical regularity on both $\mathbb{T}^{d}$ and $\mathbb{R}^{d}$.
We introduce a new infinite quadratic hierarchy that featuring a linear structure for tensor product forms, instead of marginal densities. Consequently, this newly constructed hierarchy differs from the quantum Gross-Pitaevskii hierarchy, and its structure is in fact closer to that of the classical Boltzmann hierarchy.
We prove this quadratic hierarchy admits combinatorial structures that are compatible with the bilinear $U$-$V$
 estimates we prove for the quadratic nonlinearity. These tools enable us to establish unconditional uniqueness at the critical regularity via a quadratic hierarchy approach, and thus to provide an affirmative answer to Bourgain's uniqueness concern \cite[p.152]{Bou93} for the quadratic nonlinearity in the weak-form bilinear case.

 \end{abstract}
\maketitle
\tableofcontents

%%%%无条件唯一性

\section{Introduction}
We consider the quadratic nonlinear Schr\"{o}dinger equation (NLS)
\begin{equation}\label{equ:NLS}
\begin{cases}
&i\pa_{t}u=-\Delta u+ B(u,u),\quad (t,x)\in [0,T_{0}]\times \Lambda^{d}\\
&u(0,x)=u_{0}(x)
\end{cases}
\end{equation}
where $\Lambda^{d}=\R^{d}$ or $\T^{d}$. The bilinear term is
\begin{align}\label{equ:bilinear form}
B(u,u)=a_{1}u^{2}+a_{2}|u|^{2}+a_{3}\ol{u}^{2},
\end{align}
where the coefficient $a_{i}\in \R$ for $i=1,2,3$.

On the Euclidean space, the NLS $(\ref{equ:NLS})$ enjoys the scaling invariance
\begin{align}\label{equ:nls, scaling}
u_{\la}(t,x)=\la^{2}u(\la^{2}t,\la x),\quad \la>0.
\end{align}
which preserves the homogeneous Sobolev norm $\n{u_{0}}_{\dot{H}^{s_{c}}}$ where the critical scaling exponent is given by
\begin{align}\label{equ:scaling exponent}
s_{c}:=\frac{d-4}{2}.
\end{align}
 Accordingly, the initial value problem $(\ref{equ:NLS})$ for $u_{0}\in H^{s_{c}}$ can be classified as energy subcritical, critical or
supercritical depending on whether the critical Sobolev exponent $s_{c}$ is less than, equal to, or greater than the energy exponent $s=1$.

The well-posedness theory for the NLS \eqref{equ:NLS} has been widely studied. See, for example, \cite{Bou93,Caz03,KK26,KPV96,Sta97}.
The goal of this paper is to address the uniqueness problem, especially in the energy-critical and energy-supercritical regimes where
$s_{c}\geq 1$, which corresponds to
$d\geq 6$. Our main result establishes unconditional uniqueness at the critical regularity level.
\begin{theorem}\label{thm:uniqueness for nls}
Let $s_{c}\geq 1$. There is at most one $C([0,T_{0}];H^{s_{c}}(\Lambda^{d}))$ solution to $(\ref{equ:NLS})$.
\end{theorem}
\begin{remark}
In Bourgain's seminal paper \cite{Bou93}, the well-posedness theory for the general power-type nonlinearity $|u|^{\alpha}u$ with $\alpha \geq 1$ on $\T^{d}$ is investigated. When $\alpha = 1$, this nonlinearity reduces to the quadratic case $|u|u$.
Regarding uniqueness for $\al\in[1,2)$, Bourgain remarks in \cite[p.152, Remarks (ii)]{Bou93},
\begin{align*}
\text{``We loose uniqueness here however, which is a major consideration in these problems.''}
 \end{align*}
 While Bourgain originally treated the nonlinearity $|u|^{\al}u$, the problem for the general bilinear form $B(u,u)$ remains nontrivial. On one hand, it shares the essential difficulty arising from scaling criticality. On the other hand, the $L_{t}^{1}H_{x}^{s_{c}}$ Sobolev multilinear estimates on $\T^{d}$ for the well posedness theory are actually more difficult to establish than those for higher power nonlinearities like the cubic and quintic cases, due to the fact that the extra factors present in a higher power nonlinearity allow for more freedom in the analysis.

 Here,
Theorem \ref{thm:uniqueness for nls} establishes a positive unconditional uniqueness result for the bilinear form
$B(u,u)$ in both energy-critical and energy-supercritical regimes. This theorem thus provides an affirmative answer to Bourgain's uniqueness concern within the weak-form bilinear case.
\end{remark}

The fundamental concept of unconditional uniqueness was first introduced by Kato in \cite{Kat95}.
 In the study of dispersive equations, the well-posedness theory is typically established within specific auxiliary function spaces, such as Strichartz spaces. Different choices of these auxiliary spaces may lead to distinct solutions, so that the uniqueness established is inherently construction-dependent.
 Unconditional uniqueness, by contrast, requires that for a given initial datum, the solution within its natural space is unique without relying on any auxiliary spaces. Currently, this topic has become a very active area of research. See, for example, \cite{BIT11,FPT03,GST13,Kis21unc,KOY20,MN09,MPV18,Zho97} along with the references therein for further developments regarding other dispersive equations.

On the Euclidean space $\R^{d}$, unconditional uniqueness is generally established by demonstrating that any solution must coincide with the Strichartz solution whenever the latter exists. This strategy relies heavily on inhomogeneous Strichartz estimates and has proven successful in many challenging settings such as energy-critical cases \cite{Caz03,CKSTT08,TV05}.

However, such an argument for the Euclidean setting is no longer effective when \eqref{equ:NLS} is posed on $\T^{d}$, owing to the significantly weaker dispersion in the periodic setting. Moreover, the well-posedness framework on $\T^{d}$ is substantially more intricate, frequently relying on sophisticated tools such as the
$X_{s,b}$ spaces introduced in \cite{Bou93} or the atomic $U^{p}$ and $V^{p}$ spaces \cite{HTT11,IP12}.
Consequently, addressing unconditional uniqueness problems at critical regularity on $\T^{d}$ presents a more difficult challenge and requires a new idea.

\vspace{1em}
\noindent\textbf{Novelties and Outline of the Proof}.
A key novelty of this paper is the discovery of an infinite quadratic hierarchy,
 a construction motivated by the classical Boltzmann hierarchy rather than the derivation of the quantum Gross-Pitaevskii hierarchy for the cubic or quintic nonlinearity. Crucially, we find that its inherent structure is fully compatible with bilinear $U$-$V$ estimates we establish for the quadratic nonlinearity.
This allows us to establish unconditional uniqueness at the critical regularity via a quadratic hierarchy approach.

First, we introduce a novel idea by exploiting a hierarchy structure of product form.
Specifically,
let $\phi_{1}$ and $\phi_{2}$ are two solutions to \eqref{equ:NLS} with the same initial datum.
We define the key product quantity
\begin{align}\label{equ:product quantity,intro}
\Phi^{(k)}(t):=\prod_{j=1}^{k}\phi_{1}(t,x_{j})-\prod_{j=1}^{k}\phi_{2}(t,x_{j}).
\end{align}
The advantage of this product structure is that it gives rise to an infinite quadratic hierarchy
\begin{align}\label{equ:hierarchy,quadratic,intro}
i\pa_{t}\Phi^{(k)}=\sum_{j=1}^{k}-\Delta_{x_{j}}\Phi^{(k)}+B^{(k+1)}\Phi^{(k+1)}.
\end{align}
See Section \ref{sec:Uniqueness of the Infinite Boltzmann Hierarchy} for further details.

Although quadratic NLS originates from quantum systems, the quadratic hierarchy \eqref{equ:hierarchy,quadratic,intro} obtained here is totally different from the quantum GP hierarchy, which reads
\begin{align}\label{equ:GP,hierarchy,intro}
i\pa_{t}\ga^{(k)}=\sum_{j=1}^{k}[-\Delta_{x_{j}},\ga^{(k)}]\pm \sum_{j=1}^{k}\operatorname{Tr}_{k+1}[\delta(x_{j}-x_{k+1}),\ga^{(k+1)}].
\end{align}
The GP hierarchy enjoys clear physical interpretations and is crucial in the study of Bose-Einstein condensation, see for example \cite{ESY06,ESY07,ESY09,ESY10}. However, it is inherently restricted to structures involving specific odd powers, such as cubic and quintic nonlinearities. In contrast, the quadratic hierarchy \eqref{equ:hierarchy,quadratic,intro} is structurally much closer to the classical Boltzmann hierarchy. In fact, the product formulation is directly inspired by the product structure in the Boltzmann hierarchy as analyzed in \cite{CSZ26}, which takes the form
\begin{align}\label{equ:Boltzmann,product}
f^{(k)}(t):=\prod_{j=1}^{k}f_{1}(t,x_{j},v_{j})-\prod_{j=1}^{k}f_{2}(t,x_{j},v_{j}).
\end{align}

With the newly constructed quadratic hierarchy \eqref{equ:hierarchy,quadratic,intro}, we have in fact taken a substantial step toward resolving the uniqueness problem. Different from the route of direct NLS analysis,
 we proceed with a hierarchy strategy by exploring the combinatorial structure inherent in this quadratic hierarchy \eqref{equ:hierarchy,quadratic,intro} and developing iterative estimates that are fully compatible with the combinatorial framework. Actually, the hierarchy strategy has already proved to be remarkably successful in the analysis of the GP hierarchy \eqref{equ:GP,hierarchy,intro}.

The GP hierarchy scheme originates from the uniqueness analysis of the GP hierarchy \eqref{equ:GP,hierarchy,intro} and has a physical origin in the derivation of the NLS as the mean-field limit of quantum $N$-body dynamics.
Around 2005, it was Erd\H{o}s, Schlein, and Yau who first rigorously derived the 3D cubic defocusing NLS in their fundamental papers \cite{ESY06,ESY07,ESY09,ESY10}. A key ingredient of their proof was an
 $H^{1}$-type unconditional uniqueness theorem for the GP hierarchy, established in \cite{ESY07} through a sophisticated Feynman graph analysis. The first series of ground breaking papers have motivated a large amount of work.

Subsequently in 2007,
Klainerman and Machedon \cite{KM08}, inspired by \cite{ESY07,KM93},
proved a uniqueness theorem in a Strichartz-type space under an additional a priori space-time norm condition. They provided
space-time collapsing estimates and a different combinatorial argument, the now so-called Klainerman-Machedon (KM) board game argument. Such an argument has motivated many work such as \cite{CP11,CP14der,CH16,CH16cor,CH22,CS14,GSS14,HS16,KSS11,She21,Soh15} for the uniqueness of the GP hierarchy and derivation of the NLS. Furthermore, despite having been originally tailored to NLS, it has later offered essential insights into the Boltzmann hierarchy, as demonstrated for instance in \cite{AMPT25,CDP19,CH23der,CSZ26}.

In 2013, T. Chen, Hainzl, Pavlovi{\'c}, and Seiringer \cite{CHPS15}, by introducing the quantum de Finetti theorem from \cite{LNR14} into the KM board game argument, simplified the proof of the $\R^{3}$ uncondition uniqueness theorem given in \cite{ESY07}.
Their method, integrating the KM board game argument, the quantum de Finetti theorem, and multilinear Sobolev estimates, is robust in handling unconditional uniqueness problems of GP hierarchy. See for example \cite{HTX15,HTX16}.

Somewhat unexpectedly, the uniqueness analysis of GP hierarchy, which appears more difficult than that for the NLS, began to yield novel NLS results, starting with the work of Herr and Sohinger \cite{HS19} and of X. Chen and Holmer \cite{CH19}.
Specifically, Herr and Sohinger, first discovered that the GP hierarchy method could establish new unconditional uniqueness results for the cubic NLS on $\T^{d}$, covering the full scaling-subcritical regime for $d\geq 4$. (See also \cite{Kis21} using NLS analysis in the scaling-subcritical regime.) X. Chen and Holmer, by discovering a new hierarchical uniform frequency localisation property for
the GP hierarchy, established a new uniqueness theorem for the $\T^{3}$ quintic GP hierarchy, which implies the
unconditional uniqueness result for the $\T^{3}$ quintic energy-critical NLS. In subsequent work \cite{CH22unc,CSZ22}, new combinatorics were worked out to enable the application of
$U$-$V$ multilinear estimates, which serves as a key role in completely resolving the unconditional uniqueness problems for cubic and quintic energy-critical and energy-supercritical NLS on $\R^{d}$ and $\T^{d}$.

Through a great deal of efforts of prior work, the GP hierarchy scheme within the quantum framework has emerged as a robust approach.
At present, the hierarchy scheme appears to be the only known approach to deal with the uniqueness problem on $\T^{d}$ at the critical regularity. In light of this, for the quadratic hierarchy, it is natural for us to develop this hierarchy approach, with the aim of resolving the corresponding uniqueness problem at the critical regularity.

 However, as mentioned before, the quadratic hierarchy is different from the GP hierarchy, but is closer in nature to the Boltzmann hierarchy. Hence, the combinatorial structure underlying this hierarchy needs to be reconstructed from scratch. On the other hand, although the quadratic nonlinearity looks formally simple, establishing the corresponding $L_{t}^{1}H_{x}^{s_{c}}$ Sobolev bilinear estimates
is extremely hard, and these estimates might not be true. A key idea developed in \cite{CH22unc,CSZ22} is to introduce weaker
$U$-$V$ bilinear estimates to replace the Sobolev bilinear estimates.

Based on these considerations, our proof is divided into three steps. The first step is to construct the combinatorial framework for the quadratic hierarchy and prove that it is compatible with the
$U$-$V$ estimates. The second step is to derive the iterative estimates for the hierarchy on compatible time integration domain, relying on the
$U$-$V$ bilinear estimates. The third step is to provide a complete proof of these
$U$-$V$ bilinear estimates.

The paper is organized accordingly. In Section \ref{sec:Uniqueness of the Infinite Boltzmann Hierarchy}, we introduce the quadratic hierarchy based on the product form. By applying the KM board game argument to this hierarchy, we obtain a preliminary combinatorial structure, expressed as
\begin{align}\label{equ:km,f,I,intro}
\Phi^{(1)}(t_{1})=\sum_{\mu\in \mathrm{C}^{(k)}_{\mathrm{up}} }\int_{\operatorname{Dom}(\mu)}J_{\mu}^{(k+1)}(\Phi^{(k+1)})(t_{1},\underline{t}_{k+1})
d\underline{t}_{k+1},
\end{align}
where $\mu$ is a collapsing map given by \eqref{equ:collapsing map,def}, and $\mathrm{C}^{(k)}_{\mathrm{up}}$ denotes the the class of the upper echelon form.
Nevertheless, this representation is insufficient for our purposes. Indeed, the time integration
domain $\operatorname{Dom}(\mu)$, which consists of a union of a very large
number of simplexes in high dimension, is obviously complicated for large $k$.
 In order to obtain an explicit computable expression, we introduce in Section \ref{sec:Admissible Tree} the combinatorial tree diagrams developed in \cite{CH22unc,CSZ22,CSZ26}.
 For example, as illustrated in Example \ref{example:admissible tree}, the collapsing map $\mu$, which is given by
$$
\begin{tabular}{c|ccccc}
$j$&2&3&4&5&6\\
\hline
$\mu(j)$ &1&1&1&2&5
\end{tabular}
$$
corresponds to the following admissible tree.
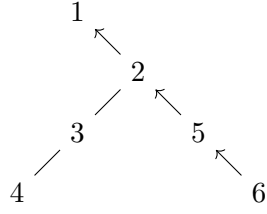
\begin{center}
\begin{tikzpicture}
\node (1) at (-0.8,0) {1};
\node (2) at (0,-0.8) {2};
\node (3) at (-0.8,-1.6) {3};
\node (4) at (-1.6,-2.4) {4};
\node (5) at (0.8,-1.6) {5};
\node (6) at (1.6,-2.4) {6};
\draw[<-] (1)--(2);
\draw[-] (2)--(3);
\draw[<-] (2)--(5);
\draw[<-] (5)--(6);
\draw[-] (3)--(4);
\end{tikzpicture}
\end{center}
 Then in Section \ref{sec:An Extended KM Board Game for the Quadratic Hierarchy}, we formulate an extended KM board game for the quadratic hierarchy, which allows us to represent the time integration domains precisely via tree diagrams, yielding
\begin{align}\label{equ:integration,upper echelon form, time integration domain,intro}
\Phi^{(1)}(t_{1})=\sum_{\mu\in \mathrm{C}^{(k)}_{\mathrm{up}} }\int_{T_{D}(\mu)}J_{\mu}^{(k+1)}(\Phi^{(k+1)})(t_{1},\underline{t}_{k+1})
d\underline{t}_{k+1},
\end{align}
where $T_{D}(\mu)$ is specified by an admissible tree diagram. It should be noted that this strategy for explicitly computing the time integration domain was first introduced in \cite{CH22unc} and is essential for proving the compatibility of the
$U$-$V$ estimates.
Although this quadratic hierarchy differs from the GP hierarchy, we find that it exhibits a compatible combinatorial structure as well, which we prove in Section \ref{sec:Compatible Time Integration Domain}.

In Section \ref{sec:Duhamel Expansion and Duhamel Tree}, we present an algorithmic construction that yields a Duhamel tree representation of the Duhamel expansion $J_{\mu}^{(k+1)}$. Continuing with the collapsing map $\mu$ from Example \ref{example:admissible tree}, the corresponding Duhamel tree takes the following form
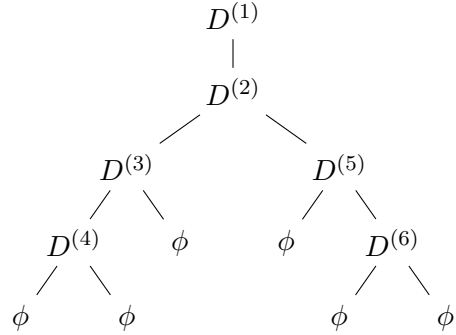
\begin{center}
\begin{tikzpicture}
\node{$D^{(1)}$}[sibling distance=80pt,level distance=1cm]
child{node{$D^{(2)}$}
child{node{$D^{(3)}$}[sibling distance=40pt,level distance=1cm]
child{node{$D^{(4)}$}
child{node{$\phi$}}
child{node{$\phi$}}}
child{node{$\phi$}}}
child{node{$D^{(5)}$}[sibling distance=40pt,level distance=1cm]
child{node{$\phi$}}
child{node{$D^{(6)}$}[sibling distance=40pt,level distance=1cm]
child{node{$\phi$}}
child{node{$\phi$}}
}}
}
;
\end{tikzpicture}
\end{center}

In Section \ref{subsection:Compatible Time Integration Domain}, we then introduce the compatible time integration domain $T_{C}(\mu)$, which is defined in terms of a Duhamel tree.
A key observation is that the Duhamel tree, after ignoring the offspring labeled $\phi$, shares the same combinatorial structure as the admissible tree. Consequently, the compatible time integration domain
$T_{C}(\mu)$ coincides exactly with $T_{D}(\mu)$.
  Building on this equivalence, we subsequently prove that the time integration domains are compatible with the space-time integration structure required by the $U$-$V$ estimates.

With this framework in place, we proceed in Section \ref{sec:Iterative Estimates} to derive the iterative estimates. The proof relies on the following high-low frequency $U$-$V$ bilinear estimates
\begin{align}
\bbn{\int_{t_{0}}^{t}e^{i(t-\tau)\Delta}(\wt{u}_{1}\wt{u}_{2})d\tau}_{X^{s}}\lesssim& \n{u_{1}}_{X^{s}}
\n{u_{2}}_{X^{\frac{d-4}{2}}}, \label{equ:bilinear estimate,high,dual argument,intro}\\
\bbn{\int_{t_{0}}^{t}e^{i(t-\tau)\Delta}(\wt{u}_{1}\wt{u}_{2})d\tau}_{X^{s}}
\lesssim& \n{u_{1}}_{X^{s}}
\lrs{T^{\frac{1}{6}}M_{0}^{\frac{1}{3}}\n{P_{\leq M_{0}}u_{2}}_{X^{\frac{d-4}{2}}}+\n{P_{>M_{0}}u_{2}}_{X^{\frac{d-4}{2}}}}. \label{equ:bilinear estimate,low,dual argument,intro}
\end{align}
where $s\in \lr{s_{c}-2,s_{c}}$.
As a preliminary step, we mark the
Duhamel tree appropriately. Then, based on this marking, we apply the above bilinear estimates at each relevant node to complete the iterative argument.

In Section \ref{sec:Bilinear Estimates}, we
give a short introduction to $U$-$V$ spaces, following the now-standard reference \cite{KTV14}, and then
complete the proof of the $U$-$V$ bilinear estimates \eqref{equ:bilinear estimate,high,dual argument,intro}--\eqref{equ:bilinear estimate,low,dual argument,intro}
via scaling-invarint Strichartz estimates based on the $l^{2}$-decoupling theorem \cite{BD15}.

 The proof of the main theorem is finally presented in Section \ref{sec:Proof of the Main Theorem}.
Additionally, we include in Appendix \ref{sec:Uniform in Time Frequency Localization for Quadratic NLS} the proof of the uniform-in-time frequency localization property for the quadratic NLS.

\section{An Infinite Quadratic Hierarchy}\label{sec:Uniqueness of the Infinite Boltzmann Hierarchy}
We start from the quadratic NLS
\begin{align}\label{equ:quadratic NLS,hierarchy}
i\pa_{t}\phi=-\Delta \phi+B(\phi,\phi).
\end{align}
We then introce an infinite quadratic hierarchy, which takes the form
\begin{align}\label{equ:quadratic hierarchy}
i\pa_{t}\Phi^{(k)}=\sum_{j=1}^{k}-\Delta_{x_{j}}\Phi^{(k)}+B^{(k+1)}\Phi^{(k+1)},
\end{align}
where the collision operator is given by
\begin{align}
B^{(k+1)}=:\sum_{j=1}^{k}B_{j,k+1}^{(k+1)}
\end{align}
 with its action on tensor products defined by
\begin{align}
B_{j,k+1}^{(k+1)}\lrs{\prod_{i=1}^{k+1}\phi_{i}(x_{i})}=\lrs{\prod_{i=1}^{j-1}\phi_{i}(x_{i})} B(\phi_{j},\phi_{k+1})(x_{j})\lrs{\prod_{i=j+1}^{k}\phi_{i}(x_{i})}.
\end{align}
One special solution to this quadratic hierarchy \eqref{equ:quadratic hierarchy} is given by the tensor product form $\phi^{\otimes k}$, where $\phi$ is the solution to \eqref{equ:quadratic NLS,hierarchy}.
Here, we do not need a well-posedness theory for the general initial data of this hierarchy. In fact, our uniqueness analysis of the hierarchy only requires restriction to the linear space generated by such product forms.

The advantage of introducing this hierarchy lies in its linearity. However, the main difficulty stems from the fact that it is a system of infinitely many coupled equations over an unbounded number of variables. To proceed,
We rewrite $\Phi^{(k)}$ in Duhamel form
\begin{align}\label{equ:duhamel,quadratic hierarchy}
\Phi^{(k)}(t)=U^{(k)}(t)\Phi^{(k)}(0)-i\int_{0}^{t}U^{(k)}(t-s)B^{(k+1)}\Phi^{(k+1)}(s)ds
\end{align}
with $U^{(k)}(t)=\prod_{j=1}^{k}e^{it\Delta_{x_{j}}}$.
Let $\phi_{1}$ and $\phi_{2}$ be two $C([0,T];H^{s_{c}})$ solutions to \eqref{equ:quadratic NLS,hierarchy} with the same initial datum.
Then we set
\begin{align*}
 \Gamma_{1}=&\lr{\phi_{1}^{\otimes k}}_{k=1}^{\infty}, \quad
 \Gamma_{2}=\lr{\phi_{2}^{\otimes k}}_{k=1}^{\infty}, \\
\Gamma=&\lr{\phi_{1}^{\otimes k}-\phi_{2}^{\otimes k}}_{k=1}^{\infty}=\lr{\Phi^{(k)}(t)=\int \phi^{\otimes k}d\nu_{t}(\phi)},
\end{align*}
where $\nu_{t}(\phi)=\delta_{\phi_{1}(t)}(\phi)-\delta_{\phi_{2}(t)}(\phi)$.

To conclude uniqueness, it suffices to prove $\Ga=0$.
Given the linearity of the infinite quadratic hierarchy \eqref{equ:duhamel,quadratic hierarchy}, $\Gamma$ is hence a solution to the
 hierarchy with zero initial datum. Therefore, iterating \eqref{equ:duhamel,quadratic hierarchy}
$k$ times allows us to write\footnote{Here, the constant $(-i)^{k}$ is omitted, as it is not important.}
\begin{align*}
\Phi^{(1)}(t_{1})=\int_{0}^{t_{1}}\int_{0}^{t_{2}}\ccc \int_{0}^{t_{k}}
J^{(k+1)}(\Phi^{(k+1)}(t_{k+1}))d\underline{t}_{k+1}
\end{align*}
where $\underline{t}_{k+1}=(t_{2},t_{3},...,t_{k+1})$ and
\begin{align*}
J^{(k+1)}(\Phi^{(k+1)})(t_{1},\underline{t}_{k+1})=U^{(1)}_{1,2}B^{(2)}U_{2,3}^{(2)}B^{(3)}\ccc
U_{k,k+1}^{(k)}B^{(k+1)}\Phi^{(k+1)},
\end{align*}
with the shorthand notation $U_{j,j+1}^{(j)}:=U^{(j)}(t_{j}-t_{j+1})$.

Since $B^{(k+1)}$ contains $k$ terms, there are $k!$ summands inside $\Phi^{(1)}(t_{1})$. More specifically, we have
\begin{align*}
\Phi^{(1)}(t_{1})=\sum_{\mu}\int_{t_{1}\geq t_{2}\geq \ccc\geq t_{k+1}}J_{\mu}^{(k+1)}(\Phi^{(k+1)})(t_{1},\underline{t}_{k+1})
d\underline{t}_{k+1},
\end{align*}
where
\begin{align}\label{equ:J,k,formula}
J_{\mu}^{(k+1)}(\Phi^{(k+1)})(t_{1},\underline{t}_{k+1})=U_{1,2}^{(1)}B_{\mu(2),2}^{(2)}U_{2,3}^{(2)}B_{\mu(3),3}^{(3)}\ccc
U_{k,k+1}^{(k)}B^{(k+1)}_{\mu(k+1),k+1}\Phi^{(k+1)}.
\end{align}
Here, $\lr{\mu}$ is a set of maps from $\lr{2,3,..,k+1}$ to $\lr{1,2,...,k}$ satisfying
that
\begin{align}\label{equ:collapsing map,def}
\text{$\mu(2)=1$ and $\mu(j)<j$ for all $j$},
\end{align}
We call it a collapsing map.

\subsection{Admissible Trees}\label{sec:Admissible Tree}
We begin with a brief review of the KM board game argument from \cite{KM08}. While initially developed for the GP hierarchy, this framework is robust and applicable to a much wider range of problems.
At the very least, it provides an effective preliminary combinatorial analysis for this quadratic hierarchy.
In short, one could sort $k!$ summands into a sum of upper echelon forms.

\begin{definition}[Upper echelon form]
Let $\mu$ be a collapsing map given by \eqref{equ:collapsing map,def}. If $\mu$ satisfies that
\begin{align}
\text{$\mu(j)\leq \mu(j+1)$ for $1\leq j\leq k-1$,}
\end{align}
 then it is in upper echelon form as they are called in \cite{KM08}.
For notional convenience, we denote
$\mathrm{C}^{(k)}_{\mathrm{up}}$ by the class of the upper echelon form.
\end{definition}

\begin{definition}[KM acceptable move and equivalent classes]
 Let $\mu$ be a collapsing map and $\sigma$ a permutation of $\{2, \ldots, k+1\}$. A Klainerman-Machedon acceptable move of $\mu$, denoted by $\operatorname{KM}(j, j+1)$, is allowed when
 $$\mu(j) \neq \mu(j+1),\quad \mu(j+1)<j.$$
The action of this move on the pair $(\mu,\sigma)$ is given by
\begin{equation}
\left\{
\begin{aligned}
\left(\mu^{\prime}, \sigma^{\prime}\right)=&\operatorname{KM}(j, j+1)(\mu, \sigma),\\
\mu^{\prime}  =&(j, j+1) \circ \mu \circ(j, j+1), \\
\sigma^{\prime}  =&(j, j+1) \circ \sigma.
\end{aligned}
\right.
\end{equation}
More generally, we say that a permutation $\rho$ is a Klainerman-Machedon acceptable move of $\mu$, if
 it admits a decomposition of the form
$$\rho=\rho_{r}\circ \rho_{r-1}\circ \ccc \circ \rho_{1}$$
 where $\rho_{1}$ is an acceptable move of $\mu$ and $\rho_{i}=(j_{i},j_{i}+1)$ is an acceptable move of $$(\mu_{i},\sigma_{i})=\mathrm{KM}(\rho_{i-1})\circ \ccc \circ \mathrm{KM}(\rho_{1})(\mu,\sigma)$$
 for $2\leq i\leq r$. The action of such a composite move is then defined by
\begin{equation}
\left\{
\begin{aligned}
(\mu',\sigma')=&\mathrm{KM}(\rho)(\mu,\sigma),\\
\mu'=&\rho \circ \mu \circ \rho^{-1},\\
\sigma'=&\rho \circ \sigma.
\end{aligned}
\right.
\end{equation}
If $\mu$ and $\mu'$ are such that there exists $\rho$ as above for which $(\mu',\sigma')=\mathrm{KM}(\rho)(\mu,\sigma)$, we say that $\mu'$ and $\mu$ are
KM-relatable, we write it as $\mu\sim \mu'$ for short. This is an equivalence relation that partitions the set
of collapsing maps into equivalence classes.
\end{definition}

Let us define
\begin{equation}
I(\mu,\sigma,\Phi^{(k+1)})
=\int_{t_{1}\geq t_{\sigma(2)}\geq \ccc \geq t_{\sigma(k+1)}}
J_{\mu}^{(k+1)}(\Phi^{(k+1)})(t_{1},\underline{t}_{k+1})d\underline{t}_{k+1}.
\end{equation}
By adapting the proof of the KM board game argument to the quadratic hierarchy, one could also have
\begin{align}\label{equ:the equality under the km move}
I(\mu,\sigma,\Phi^{(k+1)})=I(\mu',\sigma',\Phi^{(k+1)}).
\end{align}
Based on this equality \eqref{equ:the equality under the km move}, we provide the preliminary combinatorial analysis for this quadratic hierarchy in the following lemma.
\begin{lemma}[{\hspace{-0.01em}\cite{KM08}}]\label{lemma:KM board game}
For the infinite quadratic hierarchy, one also has
\begin{align}\label{equ:km,f,I}
\Phi^{(1)}(t_{1})=\sum_{\mu\in \mathrm{C}^{(k)}_{\mathrm{up}} }\int_{\operatorname{Dom}(\mu)}J_{\mu}^{(k+1)}(\Phi^{(k+1)})(t_{1},\underline{t}_{k+1})
d\underline{t}_{k+1},
\end{align}
where there are at most $4^{k}$ terms inside the the class $\mathrm{C}^{(k)}_{\mathrm{up}}$ of the upper echelon form. Here,
 the time integration domain $\operatorname{Dom}(\mu)$ is a subset of $[0,t_{1}]^{k}$, depending on $\mu$.
\end{lemma}

The time integration
domain $\operatorname{Dom}(\mu)$ is obviously complicated for large $k$, since it consists of a union of an enormous number of high-dimensional simplexes.
 To derive an explicit expression, we employ the combinatorial tree diagrams developed in \cite{CH22unc,CSZ22,CSZ26}. First,
we construct a binary tree with the following algorithm.

\begin{algorithm}[From collapsing map to admissible tree]\label{algorithm:generate an admissible tree}
~\\
\hspace*{1em}

$(1)$ Given a collapsing map $\mu$, set counter $j=1$.

$(2)$ For the current $j$, find the minimal indices
 $l$, $r$ so that $l>j$, $r>j$ and
\begin{align*}
&\mu(l)=\mu(j),\\
&\mu(r)=j,
\end{align*}
 Then assign $l$/$r$ as the left/right child of node $j$ in the tree.
If there is no such $l$/$r$, the corresponding child is absent.

$(3)$ If $j=k$, then stop. Otherwise, set $j=j+1$ and go to step $(2)$.
\end{algorithm}

We illustrate the algorithm with an example.
\begin{example}\label{example:admissible tree}
We consider the collapsing map given by
$$
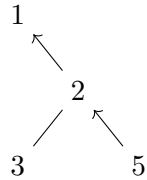
\begin{tabular}{c|ccccc}
$j$&2&3&4&5&6\\
\hline
$\mu(j)$ &1&1&1&2&5
\end{tabular}
$$
\begin{minipage}{0.3\textwidth}
\begin{center}
\begin{tikzpicture}
\node (1) at (-0.8,0) {1};
\node (2) at (0,-1) {2};
\node (3) at (-0.8,-2) {3};
\node (5) at (0.8,-2) {5};
\draw[<-] (1)--(2);
\draw[-] (2)--(3);
\draw[<-] (2)--(5);
\end{tikzpicture}
\end{center}
\end{minipage}
\begin{minipage}{0.7\textwidth}
Following Algorithm \ref{algorithm:generate an admissible tree}, we start with $j=1$ and observe that $\mu(2)=1$.  For the left and right children of node $2$, we seek the minimal $a>1$, $b>1$ such that $\mu(a)=1$, $\mu(b)=2$.  In this example, we find $a=3$, $b=5$, so we put $3$ and $5$ as the left and right children of node $2$, respectively.
\end{minipage}

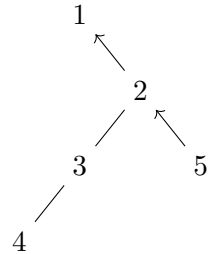
\begin{minipage}{0.3\textwidth}
\begin{center}
\begin{tikzpicture}
\node (1) at (-0.8,0) {1};
\node (2) at (0,-1) {2};
\node (3) at (-0.8,-2) {3};
\node (4) at (-1.6,-3) {4};
\node (5) at (0.8,-2) {5};
\draw[<-] (1)--(2);
\draw[-] (2)--(3);
\draw[<-] (2)--(5);
\draw[-] (3)--(4);
\end{tikzpicture}
\end{center}
\end{minipage}
\begin{minipage}{0.67\textwidth}
Next we proceed to  $j=3$. Since $\mu(3)=1$, we find the minimal $a>3$, $b>3$ such that $\mu(a)=\mu(3)=1$, $\mu(b)=3$. We find $a=4$ and while no such $b$ exists. Consequently, we put only $4$ as the left child of node $3$.
\end{minipage}

\begin{minipage}{0.67\textwidth}
We then move to $j=4$. Since there is no such $a>4$, $b>4$ satisfying $\mu(a)=\mu(4)=1$ and $\mu(b)=4$, we skip it. Moving to $j=5$, we find $\mu(6)=5$, so we put $6$ as the right child of node $5$. With all indices now present in the tree, the construction is complete and we arrive at the final tree.
\end{minipage}
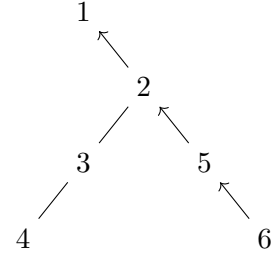
\begin{minipage}{0.3\textwidth}
\begin{center}
\begin{tikzpicture}
\node (1) at (-0.8,0) {1};
\node (2) at (0,-1) {2};
\node (3) at (-0.8,-2) {3};
\node (4) at (-1.6,-3) {4};
\node (5) at (0.8,-2) {5};
\node (6) at (1.6,-3) {6};
\draw[<-] (1)--(2);
\draw[-] (2)--(3);
\draw[<-] (2)--(5);
\draw[<-] (5)--(6);
\draw[-] (3)--(4);
\end{tikzpicture}
\end{center}
\end{minipage}
\end{example}

We now introduce the notion of an admissible tree.
\begin{definition}[Admissible tree and skeleton]
A binary tree is called an admissible tree if if the label of every child node is strictly larger than the label of its parent node. Given an admissible tree, we call, the graph of the tree without any labels in its nodes, the skeleton of the tree.
\end{definition}
For instance, the skeleton of the tree in Example \ref{example:admissible tree} is shown as follows.
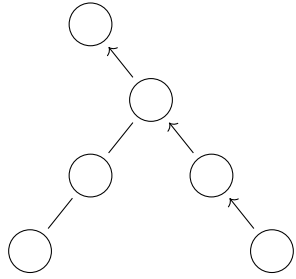
\begin{center}
\begin{tikzpicture}
\node[draw,circle,inner sep=0.2cm,outer sep=0.1cm] (1) at (-0.8,0) {};
\node[draw,circle,inner sep=0.2cm,outer sep=0.1cm] (2) at (0,-1) {};
\node[draw,circle,inner sep=0.2cm,outer sep=0.1cm] (3) at (-0.8,-2) {};
\node[draw,circle,inner sep=0.2cm,outer sep=0.1cm] (4) at (-1.6,-3) {};
\node[draw,circle,inner sep=0.2cm,outer sep=0.1cm] (5) at (0.8,-2) {};
\node[draw,circle,inner sep=0.2cm,outer sep=0.1cm] (6) at (1.6,-3) {};
\draw[<-] (1)--(2);
\draw[-] (2)--(3);
\draw[<-] (2)--(5);
\draw[<-] (5)--(6);
\draw[-] (3)--(4);
\end{tikzpicture}
\end{center}

Given an admissible binary tree, we can uniquely reconstruct a collapsing map $\mu$ that generates it. For notational convenience, we take the following notations.
\begin{equation}
\left\{
\begin{aligned}
&j\stackrel{L}{\to}i:\text{node $j$ is the left child of node $i$},\\
&j\stackrel{R}{\to}i:\text{node $j$ is the right child of node $i$},\\
&j\to i:\text{node $j$ is a child of node $i$}.
\end{aligned}
\right.
\end{equation}
The reconstruction is performed by the following algorithm.

\begin{algorithm}[From admissible tree to collapsing map]\label{algorithm:from admissible tree to collapsing map}
~\\
\hspace*{1em}

$(1)$ Given an admissible tree. Set counter $j=1$ and $\mu(2)=1$.

$(2)$ Given $j$, in the admissible tree $\al$,
\begin{align*}
&\text{if there exists $k_{1}$ such that  $k_{1}\stackrel{L}{\to} j$ in the tree, then $\mu(k_{1}):=\mu(j)$;}\\
 &\text{if there exists $k_{2}$ such that $k_{2}\stackrel{R}{\to} j$ in the tree, then $\mu(k_{2}):=j$.}
\end{align*}
Otherwise, go to step $(3)$.

$(3)$ Set $j=j+1$. If $j=k$, then stop, otherwise go to step $(2)$.
\end{algorithm}

\begin{remark}\label{remark:KM,skeleton}
Two collapsing maps $\mu$ and $\mu'$ are KM-relatable if and only if the trees corresponding to $\mu$ and $\mu'$ have the same skeleton. Moreover, if $\mu'=\mathrm{KM}(\rho)(\mu)$, then $\mathrm{Tree}(\mu')$ has the same skeleton to $\mathrm{Tree}(\mu)$ with each node
 $j$ replaced by $\rho(j)$. See \cite[Proposition 4.5]{CSZ22} for more details.
\end{remark}
\subsection{An Extended KM Board Game for the Quadratic Hierarchy}\label{sec:An Extended KM Board Game for the Quadratic Hierarchy}
We now proceed to the main part, namely how to compute time integration domain $\operatorname{Dom}(\mu)$.
We define a map $T_{D}$ which maps an admissible tree $\mathrm{Tree}(\mu)$ to a time integration domain as follows
\begin{equation} \label{equ:time integration domain generated by mu}
T_{D}(\mu)=\lr{t_{i}\geq t_{j}:j\to i\ \text{in the}\ \mathrm{Tree}(\mu)},
\end{equation}
where $j\to i$ denotes that node $i$ is a child of node $j$.

\begin{proposition}\label{lemma:upper echelon form, time integration domain}
Given a collapsing map $\mu$ in upper echelon form, we have
\begin{align*}
\sum_{\mu':\mu'\sim \mu}\int_{t_{1}\geq t_{2}\geq...\geq t_{k+1}}
J_{\mu'}^{(k+1)}(\Phi^{(k+1)})(t_{1},\underline{t}_{k+1})
d\underline{t}_{k+1}
=\int_{T_{D}(\mu)}J_{\mu}^{(k+1)}(\Phi^{(k+1)})(t_{1},\underline{t}_{k+1})
d\underline{t}_{k+1}.
\end{align*}
Consequently, we have
\begin{align}\label{equ:integration,upper echelon form, time integration domain}
\Phi^{(1)}(t_{1})=\sum_{\mu\in \mathrm{C}^{(k)}_{\mathrm{up}} }\int_{T_{D}(\mu)}J_{\mu}^{(k+1)}(\Phi^{(k+1)})(t_{1},\underline{t}_{k+1})
d\underline{t}_{k+1},
\end{align}
where there are at most $4^{k}$ terms inside the the class $\mathrm{C}^{(k)}_{\mathrm{up}}$.
\end{proposition}
\begin{proof}
Let $\Sigma(\mu)$ be the set of all acceptable moves with respect to $\mu$. Then by the equality $(\ref{equ:the equality under the km move})$, we have
\begin{align*}
\sum_{\mu':\mu'\sim \mu}I(\mu',\operatorname{Id},\Phi^{(k+1)})=&\sum_{\rho \in \Sigma(\mu)}I(\mu,\rho^{-1},\Phi^{(k+1)}).
\end{align*}

We recall that two collapsing maps $\mu$ and $\mu'$ are KM-relatable if and only if their corresponding trees share the same skeleton. This yields the characterisation
\begin{align*}
\Sigma(\mu)=\lr{\rho:\rho(i)<\rho(j),\ \text{if $j\to i$ in the $\mathrm{Tree}(\mu)$}}.
\end{align*}
Consequently, we obtain
\begin{align*}
\bigcup_{\rho\in \Sigma(\mu)}\lr{t_{1}\geq t_{\rho^{-1}(2)}\geq \ccc \geq t_{\rho^{-1}(k+1)}}=T_{D}(\mu).
\end{align*}

The desired identity \eqref{equ:integration,upper echelon form, time integration domain} then follows from
\begin{align*}
\Phi^{(1)}(t_{1})=&\sum_{\mu}\int_{t_{1}\geq t_{2}\geq \ccc\geq t_{k+1}}J_{\mu}^{(k+1)}(\Phi^{(k+1)})(t_{1},\underline{t}_{k+1})
d\underline{t}_{k+1}\\
=&\sum_{\mu\in \mathrm{C}^{(k)}_{\mathrm{up}}}\sum_{\mu':\mu'\sim \mu}\int_{t_{1}\geq t_{2}\geq...\geq t_{k+1}}
J_{\mu'}^{(k+1)}(\Phi^{(k+1)})(t_{1},\underline{t}_{k+1})
d\underline{t}_{k+1}\\
=&\sum_{\mu\in \mathrm{C}^{(k)}_{\mathrm{up}} }\int_{T_{D}(\mu)}J_{\mu}^{(k+1)}(\Phi^{(k+1)})(t_{1},\underline{t}_{k+1})
d\underline{t}_{k+1}.
\end{align*}
\end{proof}

To make this more concrete, let us examine the following example.
\begin{example}
For the map $\mu_{1}$ given in Example \ref{example:admissible tree},
there are six acceptable moves with respect to $\mu_{1}$ such that
\begin{align*}
\mu_{i}=\mathrm{KM}(\rho_{i})(\mu_{1}).
\end{align*}

\begin{table}[htbp]
\centering
\caption{Acceptable moves and Time integration domain}
\label{table:Acceptable moves and Time integration domain}
\begin{tabular}{c|ccccc||c|ccccc||c}
$j$&2&3&4&5&6& $j$&2&3&4&5&6&\text{Time integration domain} \\
\hline
$\rho_{1}(j)$ &2&$3$&$4$&$5$&6& $\rho_{1}^{-1}(j)$ &2&$3$&$4$&$5$&6&$\lr{t_{1}\geq t_{2}\geq t_{3}\geq t_{4}\geq t_{5}\geq t_{6}}$\\
$\rho_{2}(j)$ &2&$3$&$5$&$4$&6& $\rho_{2}^{-1}(j)$ &2&$3$&$5$&$4$&6&$\lr{t_{1}\geq t_{2}\geq t_{3}\geq t_{5}\geq t_{4}\geq t_{6}}$\\
$\rho_{3}(j)$ &2&$3$&$6$&$4$&5& $\rho_{3}^{-1}(j)$ &2&$3$&$5$&$6$&4&$\lr{t_{1}\geq t_{2}\geq t_{3}\geq t_{5}\geq t_{6}\geq t_{4}}$\\
$\rho_{4}(j)$ &2&$4$&$5$&$3$&6& $\rho_{4}^{-1}(j)$ &2&$5$&$3$&$4$&6&$\lr{t_{1}\geq t_{2}\geq t_{5}\geq t_{3}\geq t_{4}\geq t_{6}}$\\
$\rho_{5}(j)$ &2&$4$&$6$&$3$&5& $\rho_{5}^{-1}(j)$ &2&$5$&$3$&$6$&4&$\lr{t_{1}\geq t_{2}\geq t_{5}\geq t_{3}\geq t_{6}\geq t_{4}}$\\
$\rho_{6}(j)$ &2&$5$&$6$&$3$&4& $\rho_{6}^{-1}(j)$ &2&$5$&$6$&$3$&4&$\lr{t_{1}\geq t_{2}\geq t_{5}\geq t_{6}\geq t_{3}\geq t_{4}}$
\end{tabular}
\end{table}

The six KM-relatable maps and their associated trees are displayed below.

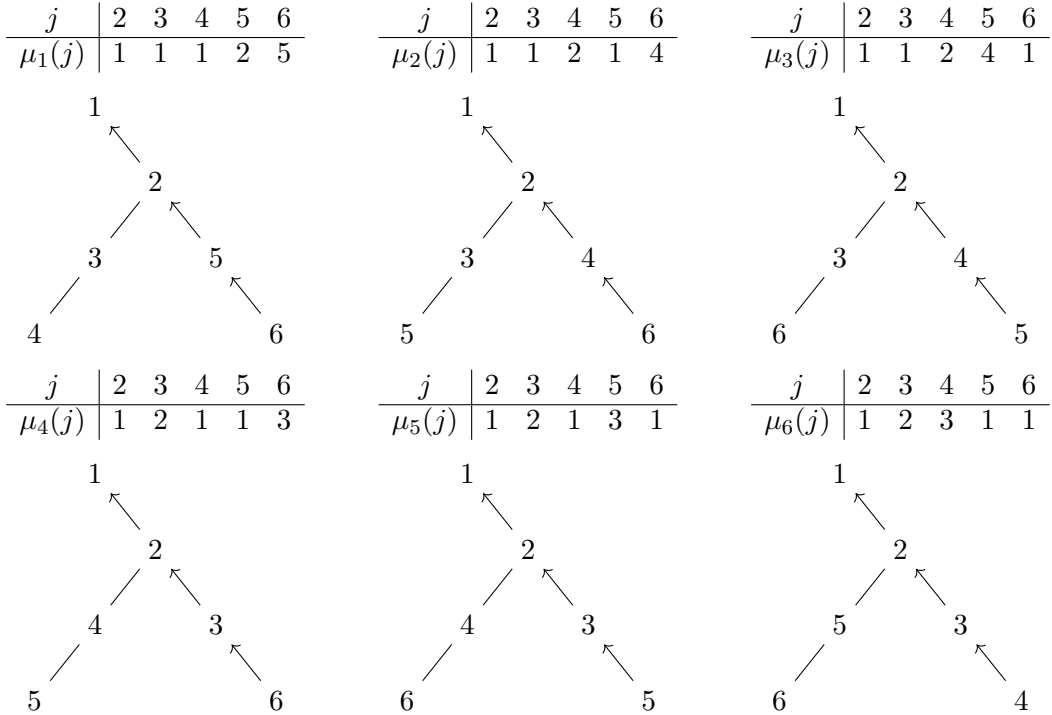
\begin{minipage}{0.32\textwidth}
\centering
$$
\begin{tabular}{c|ccccc}
$j$&2&3&4&5&6\\
\hline
$\mu_{1}(j)$ &1&1&1&2&5
\end{tabular}
$$

\begin{tikzpicture}
\node (1) at (-0.8,0) {1};
\node (2) at (0,-1) {2};
\node (3) at (-0.8,-2) {3};
\node (4) at (-1.6,-3) {4};
\node (5) at (0.8,-2) {5};
\node (6) at (1.6,-3) {6};
\draw[<-] (1)--(2);
\draw[-] (2)--(3);
\draw[<-] (2)--(5);
\draw[<-] (5)--(6);
\draw[-] (3)--(4);
\end{tikzpicture}

\end{minipage}
\begin{minipage}{0.32\textwidth}
\centering
$$
\begin{tabular}{c|ccccc}
$j$&2&3&4&5&6\\
\hline
$\mu_{2}(j)$ &1&1&2&1&4
\end{tabular}
$$

\begin{tikzpicture}
\node (1) at (-0.8,0) {1};
\node (2) at (0,-1) {2};
\node (3) at (-0.8,-2) {3};
\node (5) at (-1.6,-3) {5};
\node (4) at (0.8,-2) {4};
\node (6) at (1.6,-3) {6};
\draw[<-] (1)--(2);
\draw[-] (2)--(3);
\draw[<-] (2)--(4);
\draw[<-] (4)--(6);
\draw[-] (3)--(5);
\end{tikzpicture}

\end{minipage}
\begin{minipage}{0.32\textwidth}
\centering
$$
\begin{tabular}{c|ccccc}
$j$&2&3&4&5&6\\
\hline
$\mu_{3}(j)$ &1&1&2&4&1
\end{tabular}
$$

\begin{tikzpicture}
\node (1) at (-0.8,0) {1};
\node (2) at (0,-1) {2};
\node (3) at (-0.8,-2) {3};
\node (6) at (-1.6,-3) {6};
\node (4) at (0.8,-2) {4};
\node (5) at (1.6,-3) {5};
\draw[<-] (1)--(2);
\draw[-] (2)--(3);
\draw[<-] (2)--(4);
\draw[<-] (4)--(5);
\draw[-] (3)--(6);
\end{tikzpicture}

\end{minipage}

\begin{minipage}{0.32\textwidth}
\centering
$$
\begin{tabular}{c|ccccc}
$j$&2&3&4&5&6\\
\hline
$\mu_{4}(j)$ &1&2&1&1&3
\end{tabular}
$$

\begin{tikzpicture}
\node (1) at (-0.8,0) {1};
\node (2) at (0,-1) {2};
\node (4) at (-0.8,-2) {4};
\node (5) at (-1.6,-3) {5};
\node (3) at (0.8,-2) {3};
\node (6) at (1.6,-3) {6};
\draw[<-] (1)--(2);
\draw[-] (2)--(4);
\draw[<-] (2)--(3);
\draw[<-] (3)--(6);
\draw[-] (4)--(5);
\end{tikzpicture}

\end{minipage}
\begin{minipage}{0.32\textwidth}
\centering
$$
\begin{tabular}{c|ccccc}
$j$&2&3&4&5&6\\
\hline
$\mu_{5}(j)$ &1&2&1&3&1
\end{tabular}
$$

\begin{tikzpicture}
\node (1) at (-0.8,0) {1};
\node (2) at (0,-1) {2};
\node (4) at (-0.8,-2) {4};
\node (6) at (-1.6,-3) {6};
\node (3) at (0.8,-2) {3};
\node (5) at (1.6,-3) {5};
\draw[<-] (1)--(2);
\draw[-] (2)--(4);
\draw[<-] (2)--(3);
\draw[<-] (3)--(5);
\draw[-] (4)--(6);
\end{tikzpicture}

\end{minipage}
\begin{minipage}{0.32\textwidth}
\centering
$$
\begin{tabular}{c|ccccc}
$j$&2&3&4&5&6\\
\hline
$\mu_{6}(j)$ &1&2&3&1&1
\end{tabular}
$$

\begin{tikzpicture}
\node (1) at (-0.8,0) {1};
\node (2) at (0,-1) {2};
\node (5) at (-0.8,-2) {5};
\node (6) at (-1.6,-3) {6};
\node (3) at (0.8,-2) {3};
\node (4) at (1.6,-3) {4};
\draw[<-] (1)--(2);
\draw[-] (2)--(5);
\draw[<-] (2)--(3);
\draw[<-] (3)--(4);
\draw[-] (5)--(6);
\end{tikzpicture}

\end{minipage}

From the time integration domains listed in Table \ref{table:Acceptable moves and Time integration domain}, we observe that
\begin{align*}
&\bigcup_{\rho_{i}\in \Sigma(\mu_{1})}\lr{t_{1}\geq t_{\rho_{i}^{-1}(2)}\geq t_{\rho_{i}^{-1}(3)}\geq t_{\rho_{i}^{-1}(4)}\geq t_{\rho_{i}^{-1}(5)}\geq t_{\rho_{i}^{-1}(6)}}\\
=&\lr{t_{1}\geq t_{2},t_{2}\geq t_{3},t_{2}\geq t_{5},t_{3}\geq t_{4},t_{5}\geq t_{6}}.
\end{align*}
By the definition \eqref{equ:time integration domain generated by mu}, this union is precisely
\begin{align*}
T_{D}(\mu_{1})=\lr{t_{1}\geq t_{2},t_{2}\geq t_{3},t_{2}\geq t_{5},t_{3}\geq t_{4},t_{5}\geq t_{6}}.
\end{align*}

\end{example}

\section{Compatible Time Integration Domain}\label{sec:Compatible Time Integration Domain}
\subsection{Duhamel Expansion and Duhamel Tree}\label{sec:Duhamel Expansion and Duhamel Tree}
We initiate the analysis of the Duhamel expansion through the construction of a Duhamel tree, abbreviated as a $D$-tree. The resulting tree structure yields a diagrammatic representation of the Duhamel expansion
 $J_{\mu}^{(k+1)}$. See also \cite{CSZ26} for the Boltzmann case with the Duhamel tree diagram representation.

\begin{algorithm}[Duhamel Tree] \label{algorithm:duhamel tree}
~\\

$(1)$ Let $D^{(1)}$ be a starting node in the $D$-tree, and $D^{(2)}$ be the middle child of $D^{(1)}$.
Set counter $j=2$.

$(2)$
Given $j$,
find the minimal indices $l$ and $r$ with
$l\geq j+1$, $r\geq j+1$ such that
\begin{align*}
\mu(l)=\mu(j),
\quad \mu(r)=j.
\end{align*}
 Then place
$D^{(l)}$ as the left and and $D^{(r)}$ as the right child of $D^{(j)}$ in the $D$-tree. If there is no such $l$ or $r$,
place $\phi$ as the left or right child of $D^{(j)}$ in the $D$-tree.

$(3)$ If $j=k+1$, then stop. Otherwise, set $j=j+1$ and go to step $(2)$.
\end{algorithm}

A key observation is that the $D$-tree shares almost the same structure as the admissible tree introduced in Section \ref{sec:Admissible Tree}.

We illustrate the algorithm with the following example, which continues the one from Example \ref{example:admissible tree}.
\begin{example}\label{ex:duhamel tree}
We consider the collapsing map given by
$$
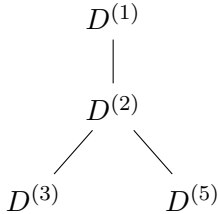
\begin{tabular}{c|ccccc}
$j$&2&3&4&5&6\\
\hline
$\mu(j)$ &1&1&1&2&5\\
\end{tabular}
$$
\begin{minipage}{0.3\textwidth}
\centering
\begin{tikzpicture}
\node{$D^{(1)}$}[sibling distance=60pt,level distance=1.2cm]
child{node{$D^{(2)}$}
child{node{$D^{(3)}$}}
child{node{$D^{(5)}$}}
};
\end{tikzpicture}
\end{minipage}
\begin{minipage}{0.67\textwidth}
Starting with $D^{(1)}$ its middle child $D^{(2)}$, we set $j=2$. The minimal indices satisfying
 $\mu(l)=1$ and $\mu(r)=2$ are $l=3$ and $r=5$. Hence, we put $D^{(3)}$ as the left child and $D^{(5)}$
as the right child of $D^{(2)}$.
\end{minipage}

\begin{minipage}{0.67\textwidth}
Next, we move to the counter $j=3$. We find that $\mu(4)=3$ and there is no $r\geq 4$ such that $\mu(r)=3$. Hence, we put $D^{(4)}$ as the left child and $\phi$ as the right child of $D^{(3)}$.
\end{minipage}
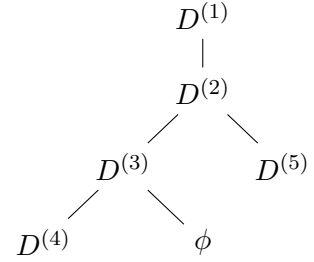
\begin{minipage}{0.3\textwidth}
\centering
\begin{tikzpicture}
\node{$D^{(1)}$}[sibling distance=60pt,level distance=1cm]
child{node{$D^{(2)}$}
child{node{$D^{(3)}$}
child{node{$D^{(4)}$}}
child{node{$\phi$}}}
child{node{$D^{(5)}$}}
}
;
\end{tikzpicture}
\end{minipage}

Finally, by iteratively applying the steps outlined in Algorithm \ref{algorithm:duhamel tree}, we construct the complete $D$-tree shown in the following tree diagram Fig. \ref{figure:duhamel tree}.
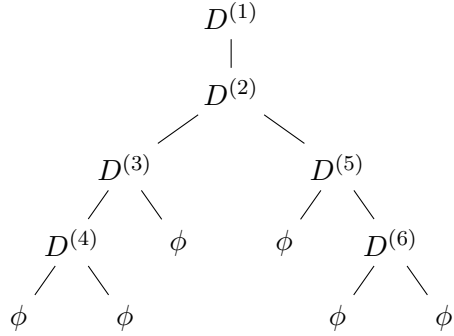
\begin{figure}[H]
\begin{tikzpicture}
\node{$D^{(1)}$}[sibling distance=80pt,level distance=1cm]
child{node{$D^{(2)}$}
child{node{$D^{(3)}$}[sibling distance=40pt,level distance=1cm]
child{node{$D^{(4)}$}
child{node{$\phi$}}
child{node{$\phi$}}}
child{node{$\phi$}}}
child{node{$D^{(5)}$}[sibling distance=40pt,level distance=1cm]
child{node{$\phi$}}
child{node{$D^{(6)}$}[sibling distance=40pt,level distance=1cm]
child{node{$\phi$}}
child{node{$\phi$}}
}}
}
;
\end{tikzpicture}
\caption{Duhamel Tree}
\label{figure:duhamel tree}
\end{figure}

From this
$D$-tree, we can directly read off the Duhamel expansion. According to \eqref{equ:J,k,formula}, we have
\begin{align}\label{equ:expression,J6}
J_{\mu}^{(6)}(\phi^{\otimes 6})(t_{1},\underline{t}_{6})=
U_{1,2}^{(1)}B_{1,2}^{(2)}U_{2,3}^{(2)}B_{1,3}^{(3)}U_{3,4}^{(3)}B_{1,4}^{(4)}
U_{4,5}^{(4)}B^{(5)}_{2,5}U_{5,6}^{(5)}B^{(6)}_{5,6}\phi^{\otimes 6}.
\end{align}
On the other hand, following the structure of the
$D$-tree \eqref{figure:duhamel tree}, we define
\begin{align*}
D^{(1)}=&U_{1}D^{(2)},\\
D^{(2)}=&U_{-2}B(U_{2}D^{(3)},U_{2}D^{(5)}),\\
D^{(3)}=&U_{-3}B(U_{3}D^{(4)},U_{3,6}\phi),\\
D^{(4)}=&U_{-4}B(U_{4,6}\phi,U_{4,6}\phi),\\
D^{(5)}=&U_{-5}B(U_{5,6}\phi,U_{5}D^{(6)}),\\
D^{(6)}=&U_{-6}B(\phi,\phi),
\end{align*}
where we have used the shorthand $U_{-i}=U(-t_{i})$ and $U_{i,j}=U(t_{i}-t_{j})$.
Then by expanding $D^{(1)}$, we get
\begin{align}
D^{(1)}=U_{1,2}B\lrc{U_{2,3}B\lrc{U_{3,4}B(U_{4,6}\phi,U_{4,6}\phi),U_{3,6}\phi},
U_{2,5}B\lrc{U_{5,6}\phi,U_{5,6}B(\phi,\phi)}},
\end{align}
which actually coincides with the expression \eqref{equ:expression,J6}. Hence, we have $J_{\mu}^{(6)}(\phi^{\otimes 6})=D^{(1)}$.
\end{example}

Next, we present a general algorithm to derive the Duhamel expansion from its corresponding $D$-tree.
\begin{algorithm}[From $D$-tree to Duhamel expansion]\label{algorithm:from d-tree to duhamel expansion}
~\\
\hspace*{1em}$(1)$  In the $D$-tree, we replace $\phi$ by $U_{-k-1}\phi$ and set
\begin{align*}
&D^{(k+1)}=U_{-k-1}B(\phi,\phi).
\end{align*}
 Set counter $j=k$.

$(2)$ Given $j$, set
\begin{align*}
&D^{(j)}=U_{-j}B(U_{j}C_{r},U_{j}C_{l}),
\end{align*}
where $C_{l}$ is the left child and $C_{r}$ is the right child of $D^{(l)}$ in the updated $D$-tree.

$(3)$ Set $j=j-1$. If $j=1$, set
\begin{align*}
&D^{(1)}=U_{1}D^{(2)},
\end{align*}
and stop, otherwise go to step $(2)$.

\end{algorithm}

By employing Algorithm \ref{algorithm:from d-tree to duhamel expansion}, we derive the following proposition.
\begin{proposition}\label{prop:from d-tree to duhamel expansion}
There holds that
\begin{align}\label{equ:from d-tree to duhamel expansion}
J_{\mu}^{(k+1)}(\phi^{\otimes (k+1)})(t_{1},\underline{t}_{k+1})=D^{(1)}(t_{1},\underline{t}_{k+1}).
\end{align}
\end{proposition}

\subsection{Compatible Time Integration Domain}\label{subsection:Compatible Time Integration Domain}
The concept of a compatible time integration domain was initially proposed in \cite{CH22unc} for the purpose of applying $U$-$V$ space techniques to cubic nonlinearities. For the quadratic hierarchy studied here, the structure differs from that of the cubic Gross-Pitaevskii hierarchy. We find that,
without additional operations on the tree structure, the time integration domain is indeed compatible with the subsequent $U$-$V$ bilinear estimates.
We first introduce the concept of the compatible time integration domain
  in terms of the Duhamel trees and illustrate it through an example.

\begin{definition}\label{definition:compatible time integration domain}
We define the compatible time integration domain as follows
\begin{align}\label{equ:compatible time integration domain}
T_{C}(\mu)=\lr{t_{i}\geq t_{j}:D^{(j)}\to D^{(i)}}
\end{align}
where $D^{(j)}\to D^{(i)}$ denotes that $D^{(j)}$ is the child of $D^{(i)}$. Moreover, we say that $D^{(l)}$ is the offspring of $D^{(j)}$ if there exist $l_{1}$,...,$l_{r}$ such that
$$D^{(l)}\to D^{(l_{1})}\to \ccc \to D^{(l_{r})}\to D^{(j)}.$$
\end{definition}
\begin{remark}
Since the Duhamel tree, after ignoring the offspring labeled $\phi$, shares the same structure as the admissible tree, the compatible time integration domain is exactly $T_{D}(\mu)$. In other words, we have
\begin{align}\label{equ:compatible,Tc,Td}
T_{C}(\mu)=T_{D}(\mu).
\end{align}
\end{remark}

\begin{example} \label{example:duhamel integral with the compatible time domain}
We continue with Example \ref{ex:duhamel tree} and compute
\begin{align*}
\int_{T_{C}(\mu)}J_{\mu}^{(6)}(\phi^{\otimes 6})(t_{1},\underline{t}_{6})d\underline{t}_{6}.
\end{align*}
From the $D$-tree in Fig. $\ref{figure:duhamel tree}$, the compatible time integration domain is given by
\begin{align*}
\int_{t_{2}=0}^{t_{1}}\int_{t_{3}=0}^{t_{2}}
\int_{t_{5}=0}^{t_{2}}\int_{t_{4}=0}^{t_{3}}\int_{t_{6}=0}^{t_{5}}.
\end{align*}
We write $\int_{t_{6}=0}^{t_{1}}$ on the outside and hence $\int_{t_{5}=0}^{t_{1}}$
changes into $\int_{t_{5}=t_{6}}^{t_{1}}$. Then the domain becomes
\begin{align*}
\int_{t_{6}=0}^{t_{1}}\int_{t_{2}=t_{6}}^{t_{1}}\int_{t_{3}=0}^{t_{2}}
\int_{t_{5}=t_{6}}^{t_{2}}\int_{t_{4}=0}^{t_{3}}.
\end{align*}
Thus, we can rewrite the integral as
\begin{align*}
&\int_{T_{C}(\mu)}J_{\mu}^{(6)}(\phi^{\otimes 6})(t_{1},\underline{t}_{6})d\underline{t}_{6}\\
=&\int_{t_{6}=0}^{t_{1}}
\int_{t_{2}=t_{6}}^{t_{1}}\int_{t_{3}=0}^{t_{2}}
\int_{t_{5}=t_{6}}^{t_{2}}\int_{t_{4}=0}^{t_{3}}U_{1}D^{(2)}d\underline{t}_{6}\\
=&\int_{t_{6}=0}^{t_{1}}
\int_{t_{2}=t_{6}}^{t_{1}}U_{1,2}B\lrs{\int_{t_{3}=0}^{t_{2}}
\int_{t_{4}=0}^{t_{3}}U_{2}D^{(3)},\int_{t_{5}=t_{6}}^{t_{2}}U_{2}D^{(5)}}d\underline{t}_{6}\\
=&\int_{t_{6}=0}^{t_{1}}
\int_{t_{2}=t_{6}}^{t_{1}}U_{1,2}B\lrc{\int_{t_{3}=0}^{t_{2}}
U_{2,3}B\lrs{\int_{t_{4}=0}^{t_{3}}U_{3}D^{(4)},U_{3,6}\phi},\int_{t_{5}=t_{6}}^{t_{2}}U_{2,5}
B(U_{5,6}\phi,U_{5}D^{(6)})}d\underline{t}_{6}
\end{align*}
where $D^{(j)}$ is shown in Example \ref{ex:duhamel tree}. We can see that all the Duhamel structures are fully compatible with $U$-$V$ multilinear estimates, which we will show in Section \ref{sec:Iterative Estimates}.
\end{example}

We now state a general algorithm that performs the Duhamel integration over the compatible domain.
\begin{algorithm}[From $D$-tree to Duhamel integration]\label{algorithm:from d-tree to duhamel integration}
~\\
\hspace*{1em}$(1)$ Let
$$Q^{(k+1)}(t_{k+1})=U_{-k-1}D^{(k+1)}(t_{k+1}),$$
and replace $D^{(k+1)}$ by $Q^{(k+1)}(t_{k+1})$ in the $D$-tree.

$(2)$ Set counter $j=k$.

$(3)$ Given $j$, there exists only one $i$ such that $D^{(j)}\to D^{(i)}$. Then there will be two cases as follows.

Case 1. $D^{(k+1)}$ is the offspring of $D^{(j)}$. Then set
\begin{align*}
Q^{(j)}(t_{i},t_{k+1})
=&\int_{t_{j}=t_{k+1}}^{t_{i}}U_{-j}B\lrs{U_{j}C_{l},U_{j}C_{r}
}dt_{j}.
\end{align*}

Case 2. $D^{(k+1)}$ is not the offspring of $D^{(l)}$. Then set
\begin{align*}
Q^{(j)}(t_{i},t_{k+1})
=&\int_{t_{j}=0}^{t_{i}}
U_{-j}B\lrs{U_{j}C_{l},U_{j}C_{r}
}dt_{j}.
\end{align*}
Here, $C_{l}$ or $C_{r}$ is the left or right child of $D^{(j)}$ in the updated $D$-tree.

$(4)$ Update the $D$-tree by using $Q^{(j)}(t_{j+1},t_{k+1})$ to replace $D^{(j)}$.

$(5)$ Set $j=j-1$.
If $j>1$, go to step $(3)$.
If $j=1$, set
$$Q^{(1)}(t_{1})=\int_{t_{k+1}=0}^{t_{1}}U_{1}Q^{(2)}(t_{2},t_{k+1})dt_{k+1},$$
and stop.
\end{algorithm}
Consequently, we arrive at the following representation.
\begin{proposition}\label{prop:from d-tree to duhamel integration}
There holds that
\begin{align}
\int_{T_{C}(\mu)}J_{\mu}^{(k+1)}(\phi^{\otimes (k+1)})
(t_{1},\underline{t}_{k+1})d\underline{t}_{k+1}=Q^{(1)}(t_{1}).
\end{align}
\end{proposition}
\begin{proof}
It follows from Algorithm $\ref{algorithm:from d-tree to duhamel integration}$.
\end{proof}

\section{Iterative Estimates for the Hierarchy}\label{sec:Iterative Estimates}
In this section, we establish iterative estimates for the hierarchy over the compatible time integration domain. These estimates rely on the $U$-$V$ bilinear estimates. As a preliminary step, we introduce a marking algorithm on the $D$-tree, which will subsequently guide the application of the bilinear estimates.
\begin{algorithm}[Marked $D$-Tree] \label{algorithm:marked d-tree}
~\\
\hspace*{1em}$(1)$ We put a subscript $R$ at $D^{(k+1)}$, writing $D_{R}^{(k+1)}$. Here, we use the subscript $R$ to denote the roughest term.

$(2)$ Set counter $l=k-1$. If $D^{(k+1)}$ is an offspring (see Definition \ref{definition:compatible time integration domain}) of $D^{(l)}$, we place a subscript $R$ on $D^{(l)}$, i.e. $D_{R}^{(l)}$. Moreover, if one of the children of $D^{(l)}$ is $\phi$, we place a subscript $\phi$ on $D^{(l)}$, yielding either $D_{\phi}^{(l)}$ or $D_{\phi,R}^{(l)}$. This procedure gives rise to four possible markings, as follows.

Case 1. $D^{(l)}$: when $\phi$ is not a child of $D^{(l)}$, and $D^{(k+1)}$ is not an offspring of $D^{(l)}$.

Case 2. $D_{R}^{(l)}$: when $\phi$ is not a child of $D^{(l)}$, and $D^{(k+1)}$ is an offspring of $D^{(l)}$.

Case 3. $D_{\phi}^{(l)}$: when $\phi$ is a child of $D^{(l)}$, and $D^{(k+1)}$ is not an offspring of $D^{(l)}$.

Case 4. $D_{\phi,R}^{(l)}$: when $\phi$ is a child of $D^{(l)}$, and $D^{(k+1)}$ is an offspring of $D^{(l)}$.

$(3)$ Set $l=l-1$. If $l=0$, then stop, otherwise go to step $(2)$.
\end{algorithm}
\begin{remark}
Since there are $(k+1)$ copies of $\phi$, the number of nodes of the form $D_{\phi,R}^{(l)}$ and $D_{\phi}^{(l)}$ is at least $\frac{k}{2}$.
\end{remark}

We illustrate the iterative estimate part using the Duhamel expansion from Example \ref{example:duhamel integral with the compatible time domain} over its compatible time integration domain.
\begin{example} \label{example:estimate part}
Applying Algorithm \ref{algorithm:marked d-tree} to the $D$-tree in Fig. \ref{figure:duhamel tree} yields the marked tree shown in Fig. \ref{figure:marked duhamel tree}.
\begin{figure}[H]
\begin{tikzpicture}
\node{$D^{(1)}$}[sibling distance=80pt,level distance=1cm]
child{node{$D_{R}^{(2)}$}
child{node{$D_{\phi}^{(3)}$}[sibling distance=40pt,level distance=1cm]
child{node{$D_{\phi}^{(4)}$}
child{node{$\phi$}}
child{node{$\phi$}}}
child{node{$\phi$}}}
child{node{$D_{\phi,R}^{(5)}$}[sibling distance=40pt,level distance=1cm]
child{node{$\phi$}}
child{node{$D_{R}^{(6)}$}[sibling distance=40pt,level distance=1cm]
child{node{$\phi$}}
child{node{$\phi$}}
}}
}
;
\end{tikzpicture}
\caption{Marked Duhamel Tree}
\label{figure:marked duhamel tree}
\end{figure}
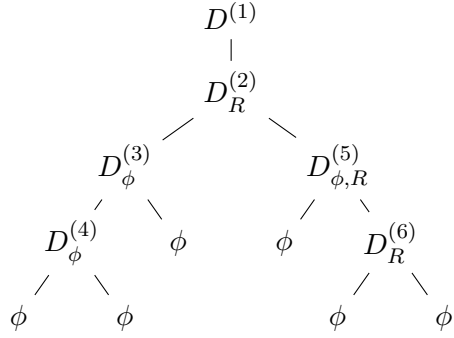
Next, we apply Algorithm \ref{algorithm:from d-tree to duhamel integration} to obtain
\begin{align*}
Q^{(1)}(t_{1})=&\int_{t_{6}=0}^{t_{1}}U_{1}Q^{(2)}(t_{1},t_{6})dt_{6},\\
Q^{(2)}(t_{1},t_{6})=&\int_{t_{2}=t_{6}}^{t_{1}}U_{-2}B(U_{2}Q^{(3)},U_{2}Q^{(5)})dt_{2},\\
Q^{(3)}(t_{2},t_{6})=&\int_{t_{3}=0}^{t_{2}}U_{-3}B(U_{3}Q^{(4)},U_{3,6}\phi)dt_{3},\\
Q^{(4)}(t_{3},t_{6})=&\int_{t_{4}=0}^{t_{3}}U_{-4}B(U_{4,6}\phi,U_{4,6}\phi)dt_{4},\\
Q^{(5)}(t_{2},t_{6})=&\int_{t_{5}=t_{6}}^{t_{2}}U_{-5}B(U_{5,6}\phi,U_{5}D^{(6)})dt_{5},\\
Q^{(6)}(t_{6})=&U_{-6}B(\phi,\phi).
\end{align*}
Using the Minkowski inequality and estimate \eqref{equ:u-v and sobolev}, we obtain
\begin{align*}
\n{\lra{\nabla_{x_{1}}}^{
\frac{d-8}{2}}Q^{(1)}(t_{1})}
_{L_{T}^{\wq}L_{x_{1}}^{2}}
=&\bbn{\int_{t_{6}=0}^{t_{1}}\lra{\nabla_{x_{1}}}^{
\frac{d-8}{2}}U_{1}Q^{(2)}(t_{1},t_{6})dt_{6}}_{L_{T}^{\wq}L_{x_{1}}^{2}}\\
\leq&\int_{0}^{T}\bbn{\lra{\nabla_{x_{1}}}^{
\frac{d-8}{2}}U_{1}Q^{(2)}(t_{1},t_{6})}_{L_{t_{1}}^{\wq}L_{x_{1}}^{2}}dt_{6}\\
\leq&\int_{0}^{T}\bbn{U_{1}Q^{(2)}(t_{1},t_{6})}_{X^{\frac{d-8}{2}}}dt_{6}.
\end{align*}
Since $D_{R}^{(2)}$ carries the subscript $R$, we apply the $U$-$V$ bilinear estimate \eqref{equ:bilinear estimate,high,dual argument} with $s=\frac{d-8}{2}$ to get
\begin{align*}
\bbn{U_{1}Q^{(2)}(t_{1},t_{6})}_{X^{\frac{d-8}{2}}} \leq& \bbn{U_{1}\int_{t_{2}=t_{6}}^{t_{1}}U_{-2}B(U_{2}Q^{(3)},U_{2}Q^{(5)})dt_{2}}_{X^{\frac{d-8}{2}}}\\
\leq&C\n{U_{2}Q^{(3)}}_{X^{\frac{d-4}{2}}}\n{U_{2}Q^{(5)}}_{X^{\frac{d-8}{2}}}.
\end{align*}

For the $\phi$-marked nodes $D_{\phi}^{(3)}$ and $D_{\phi}^{(4)}$, we apply the low-frequency estimate \eqref{equ:bilinear estimate,low,dual argument} with $s=\frac{d-4}{2}$. This yields
\begin{align*}
\bbn{U_{2}Q^{(3)}}_{X^{\frac{d-4}{2}}}
\leq&C\n{U_{3}Q^{(4)}}_{X^{\frac{d-4}{2}}}\lrs{T^{\frac{1}{6}}M_{0}^{\frac{1}{3}}\n{P_{\leq M_{0}}\phi}_{H^{\frac{d-4}{2}}}+\n{P_{>M_{0}}\phi}_{H^{\frac{d-4}{2}}}},\\
\bbn{U_{3}Q^{(4)}}_{X^{\frac{d-4}{2}}}
\leq&C\n{U_{3,6}\phi}_{X^{\frac{d-4}{2}}}\lrs{T^{\frac{1}{6}}M_{0}^{\frac{1}{3}}\n{P_{\leq M_{0}}\phi}_{H^{\frac{d-4}{2}}}+\n{P_{>M_{0}}\phi}_{H^{\frac{d-4}{2}}}}.
\end{align*}

The node $D_{\phi,R}^{(5)}$ carries both markings, so we use the low-frequency estimate \eqref{equ:bilinear estimate,low,dual argument} with $s=\frac{d-8}{2}$ to obtain
\begin{align*}
\n{U_{2}Q^{(5)}}_{X^{\frac{d-8}{2}}}\leq C
\n{U_{5}Q^{(6)}}_{X^{\frac{d-8}{2}}}\lrs{T^{\frac{1}{6}}M_{0}^{\frac{1}{3}}\n{P_{\leq M_{0}}\phi}_{H^{\frac{d-4}{2}}}+\n{P_{>M_{0}}\phi}_{H^{\frac{d-4}{2}}}}.
\end{align*}

Finally, for the roughest term $Q^{(6)}$, we use Sobolev inequality to get
\begin{align*}
\n{U_{5}Q^{(6)}}_{X^{\frac{d-8}{2}}}\leq \n{B(\phi,\phi)}_{H^{\frac{d-8}{2}}}\leq C\n{\phi}_{H^{\frac{d-4}{2}}}^{2}.
\end{align*}

Collecting these above estimates, we arrive at
\begin{align*}
&\n{\lra{\nabla_{x_{1}}}^{
\frac{d-8}{2}}Q^{(1)}(t_{1})}
_{L_{T}^{\wq}L_{x_{1}}^{2}}\\
\leq & C^{5}\int_{0}^{T}\n{\phi}_{H^{\frac{d-4}{2}}}^{3}\lrs{T^{\frac{1}{6}}M_{0}^{\frac{1}{3}}\n{P_{\leq M_{0}}\phi}_{H^{\frac{d-4}{2}}}+\n{P_{>M_{0}}\phi}_{H^{\frac{d-4}{2}}}}^{3}dt_{6}.
\end{align*}
\end{example}

Building on the above example, we now establish iterative estimates for the general case.
\begin{proposition}\label{lemma:estimate for the compatible time integration domain}
Let $\Phi^{(k)}(t)=\int \phi^{\otimes k}d\nu_{t}(\phi)$,
where $\nu_{t}(\phi)=\delta_{\phi_{1}(t)}(\phi)-\delta_{\phi_{2}(t)}(\phi)$.
 Then we have
\begin{align}\label{equ:estimate for the compatible time integration domain}
&\bbn{\lra{\nabla_{x_{1}}}^{\frac{d-8}{2}}\int_{T_{C}(\mu)}J_{\mu}^{(k+1)}
(\Phi^{(k+1)})
(t_{1},\underline{t}_{k+1})d\underline{t}_{k+1}}
_{L_{T}^{\wq}L_{x_{1}}^{2}}\\
\leq &C^{k}\int_{0}^{T}\int \n{\phi}_{H^{\frac{d-4}{2}}}^{\frac{k}{2}+1}\lrs{T^{\frac{1}{6}}M_{0}^{\frac{1}{3}}\n{P_{\leq M_{0}}\phi}_{H^{\frac{d-4}{2}}}+\n{P_{>M_{0}}\phi}_{H^{\frac{d-4}{2}}}}^{\frac{k}{2}}
d|\nu_{t_{k+1}}|(\phi)dt_{k+1}.\notag
\end{align}
\end{proposition}
\begin{proof}
By Proposition $\ref{prop:from d-tree to duhamel integration}$, we rewrite
\begin{align*}
I:=&\bbn{\lra{\nabla_{x_{1}}}^{\frac{d-8}{2}}\int_{T_{C}(\mu)}J_{\mu}^{(k+1)}
(\Phi^{(k+1)})
(t_{1},\underline{t}_{k+1})d\underline{t}_{k+1}}
_{L_{T}^{\wq}L_{x_{1}}^{2}}\\
=&\bbn{\lra{\nabla_{x_{1}}}^{\frac{d-8}{2}}\int_{T_{C}(\mu)}\int J_{\mu}^{(k+1)}
(\phi^{\otimes (k+1)})
(t_{1},\underline{t}_{k+1})d\nu_{t_{k+1}}(\phi)d\underline{t}_{k+1}}
_{L_{T}^{\wq}L_{x_{1}}^{2}}\\
=&\bbn{\lra{\nabla_{x_{1}}}^{\frac{d-8}{2}} \int_{0}^{t_{1}}\int U_{1}Q^{(2)}(t_{1},t_{k+1})d\nu_{t_{k+1}}(\phi)dt_{k+1}}
_{L_{T}^{\wq}L_{x_{1}}^{2}}.
\end{align*}
Applying Minkowski's inequality together with estimate \eqref{equ:u-v and sobolev}, we obtain
\begin{align*}
I\leq& \int_{0}^{T}\int \n{U_{1}Q^{(2)}}_{X^{\frac{d-8}{2}}}d|\nu_{t_{k+1}}| (\phi)dt_{k+1}.
\end{align*}
Thus, it suffices to bound $\n{U_{1}Q^{(2)}}_{X^{\frac{d-8}{2}}}$. We achieve this by the following algorithm.
\begin{algorithm}[Iterative Estimates]
~\\
\hspace*{1em}$(1)$ Set counter $j=2$.

$(2)$ Given $j$, there exists a unique index $i$ such that $Q^{(j)}\to Q^{(i)}$. We need to estimate the two quantities
\begin{align}
\bbn{\int_{a}^{t_{i}}U_{i,j}B\lrs{U_{j}C_{l},U_{j}C_{r}
}dt_{j}}_{X^{\frac{d-8}{2}}},\\
\bbn{\int_{a}^{t_{i}}U_{i,j}B\lrs{U_{j}C_{l},U_{j}C_{r}
}dt_{j}}_{X^{\frac{d-4}{2}}},
\end{align}
where $C_{r}$ and $C_{l}$ are the left and right child of $Q^{(j)}$ and $a\in \lr{0,t_{k+1}}$.

For each node $j$, the marking on $D^{(j)}$ leads to four possible treatments.

Case $1$. $D^{(j)}$. We apply estimate \eqref{equ:bilinear estimate,high,dual argument} with $s=\frac{d-4}{2}$.

Case $2$. $D_{R}^{(j)}$. We apply estimate \eqref{equ:bilinear estimate,high,dual argument} with $s=\frac{d-8}{2}$, and place the $R$-carrying child in $X^{\frac{d-8}{2}}$ norm.

Case $3$. $D_{\phi}^{(j)}$. We apply estimate \eqref{equ:bilinear estimate,low,dual argument} with $s=\frac{d-4}{2}$ and replace the resulting $\n{U\phi}_{X^{\frac{d-4}{2}}}$ by $\n{\phi}_{H^{\frac{d-4}{2}}}$.

Case $4$. $D_{\phi,R}^{(j)}$.
We apply estimate \eqref{equ:bilinear estimate,low,dual argument} with $s=\frac{d-8}{2}$, place $R$-carrying child in $X^{\frac{d-8}{2}}$ norm, and replace the resulting $\n{U\phi}_{X^{\frac{d-4}{2}}}$ by $\n{\phi}_{H^{\frac{d-4}{2}}}$.

$(3)$ Set counter $j=j+1$. If $j<k$, go to step $(2)$, otherwise go to step $(4)$.

$(4)$ We are now at the $k$-th coupling and would have applied \eqref{equ:bilinear estimate,low,dual argument} at least $\frac{k}{2}$ times, so we arrive at
\begin{align*}
&\bbn{\lra{\nabla_{x_{1}}}^{\frac{d-8}{2}}\int_{T_{C}}
J_{\mu}^{(k+1)}(\Phi^{(k+1)})
(t_{1},\underline{t}_{k+1})d\underline{t}_{k+1}}
_{L_{T}^{\wq}L_{x_{1}}^{2}}\\
\leq &C^{k-1}\int_{0}^{T}\int\n{\phi}_{H^{\frac{d-4}{2}}}^{\frac{k}{2}-1}\lrs{T^{\frac{1}{6}}M_{0}^{\frac{1}{3}}\n{P_{\leq M_{0}}\phi}_{H^{\frac{d-4}{2}}}+\n{P_{>M_{0}}\phi}_{H^{\frac{d-4}{2}}}}^{\frac{k}{2}}\\
&\quad \quad \quad \quad\quad \n{B(\phi,\phi)}_{H^{\frac{d-8}{2}}}d|\nu_{t_{k+1}}|(\phi)dt_{k+1}
\end{align*}
Finally, applying Sobolev inequality to $\n{B(\phi,\phi)}_{H^{\frac{d-8}{2}}}$, we arrive at
\begin{align*}
\leq C^{k}\int_{0}^{T}\int \n{\phi}_{H^{\frac{d-4}{2}}}^{\frac{k}{2}+1}\lrs{T^{\frac{1}{6}}M_{0}^{\frac{1}{3}}\n{P_{\leq M_{0}}\phi}_{H^{\frac{d-4}{2}}}+\n{P_{>M_{0}}\phi}_{H^{\frac{d-4}{2}}}}^{\frac{k}{2}}
d|\nu_{t_{k+1}}|(\phi)dt_{k+1}.
\end{align*}
\end{algorithm}
Hence, we complete the proof.
\end{proof}

\section{$U$-$V$ Bilinear Estimates}\label{sec:Bilinear Estimates}
We now recall the definition of $U_{t}^{p}$ and $V_{t}^{p}$ as in the standard reference \cite{KTV14}.
 Here, $V_{t}^{p}$ is the space of functions of bounded $p$-variation of Wiener, and
the atomic $U_{t}^{p}$ space is introduced by Koch and Tataru \cite{KT05,KT07}.
The spaces $X^{s}([0,T))$ and $Y^{s}([0,T))$ are defined as the sets of functions
 $u:[0,T)\mapsto H^{s}(\T^{d})$ such that for every $\xi\in \Z^{d}$ the map $t\mapsto \widehat{e^{-it\Delta}u(t)}(\xi)$ is in $U^{2}([0,T);\C)$ and $V_{rc}^{2}([0,T);\C)$, respectively. Their norms are given by
\begin{align*}
\n{u}_{X^{s}([0,T))}:=\lrs{\sum_{\xi\in \Z^{d}}\lra{\xi}^{2s}
\n{\widehat{e^{-it\Delta}u(t)}(\xi)}_{U^{2}}^{2}}^{1/2},\\
\n{u}_{Y^{s}([0,T))}:=\lrs{\sum_{\xi\in \Z^{d}}\lra{\xi}^{2s}
\n{\widehat{e^{-it\Delta}u(t)}(\xi)}_{V^{2}}^{2}}^{1/2}.
\end{align*}
For further details, we refer to \cite{HTT11,HTT14,IP12,KV16}. The main properties are summarised in the following lemma.
\begin{lemma}[{\hspace{-0.01em}\cite[Propositions 2.8--2.11]{HTT11} }]\label{lemma:U-V estimate,dual argument}
There holds that
\begin{align}
&\n{u}_{L_{t}^{\wq}H_{x}^{s}}\lesssim \n{u}_{Y^{s}}\lesssim \n{u}_{X^{s}},\label{equ:u-v and sobolev}\\
&\n{e^{it\Delta}f}_{Y^{s}}\leq  \n{e^{it\Delta}f}_{X^{s}}\leq\n{f}_{H^{s}}, \label{equ:u-v and sobolev, special case}
\end{align}
 For $f\in L^{1}(0,T;H^{s}(\T^{d}))$, we have
\begin{align}
&\bbn{\int_{t_{0}}^{t}e^{i(t-\tau)\Delta}f(\tau,\cdot)d\tau}_{X^{s}([0,T))}
&\leq \sup_{g\in Y^{-s}([0,T)):\n{g}_{Y^{-s}}=1}\bbabs{\int_{0}^{T}\int_{\T^{d}}
f(t,x)\ol{g(t,x)}dtdx},
\end{align}
for all $t_{0}\in[0,T)$.
\end{lemma}

To establish the scaling-critical $U$-$V$ bilinear estimates, we rely on the Strichartz estimates on the torus and
 a bilinear Strichartz estimate.
\begin{lemma}[Strichartz estimate on $\T^{d}$\cite{BD15,KV16}]\label{lemma:Strichartz estimate}
For $p>\frac{2(d+2)}{d},$
\begin{align}\label{equ:Strichartz estimate}
\n{P_{\leq M}u}_{L_{t,x}^{p}}
\lesssim M^{\frac{d}{2}-\frac{d+2}{p}}\n{P_{\leq M}u}_{Y^{0}([0,T))}.
\end{align}
Further, let $M$ be a dyadic value and let $Q$ be a noncentered $M$-cube
in Fourier space
$$Q=\lr{\xi_{0}+\eta:|\eta|<M}.$$
Let $P_{Q}$ be the corresponding Littlewood-Paley projection, then by the
Galilean invariance, we have
\begin{align}\label{equ:strichartz estimate with noncertered frequency localization}
\n{P_{Q}u}_{L_{t,x}^{p}}\lesssim M^{\frac{d}{2}-\frac{d+2}{p}}
\n{P_{Q}u}_{Y^{0}([0,T))}.
\end{align}
\end{lemma}

\begin{lemma}[{\hspace{-0.01em}\cite{HTT11,KV16}}]\label{lemma:bilinear,strchartz,L2}
For \(d \ge 3\), when \(N_1 \ge N_2\), there exists \(\delta > 0\) such that
\begin{align}\label{equ:bilinear estimate,Strichartz,L2}
\| \lrs{P_{N_1} u_{1}}  \lrs{P_{N_2} u_{2}} \|_{L_{t,x}^2([0,1] \times \mathbb{T}^d)}
\lesssim  N_2^{\frac{d-2}{2}} \left( \frac{N_2}{N_1} + \frac{1}{N_2} \right)^{\delta}
\| u_{1} \|_{Y^0} \| u_{2} \|_{Y^0}.
\end{align}
\end{lemma}

To handle the quadratic nonlinearities in the proof, we specifically need an $L^{\frac{3}{2}}$ bilinear estimate with noncentered frequency localization.
\begin{corollary}
Let $Q$ be a noncentered $N_{2}$-cube.
For $d\geq 6$, when \(N_1 \ge N_2\), there exists \(\delta > 0\) such that
\begin{align}\label{equ:bilinear estimate,Strichartz,L32}
\|\lrs{P_{Q} P_{N_1} u_{1}}  \lrs{P_{N_2} u_{2}} \|_{L_{t,x}^{\frac{3}{2}}([0,1] \times \mathbb{T}^d)}
\lesssim  N_2^{\frac{d-4}{3}} \left( \frac{N_2}{N_1} + \frac{1}{N_2} \right)^{\delta}
\| u_{1} \|_{Y^0} \| u_{2} \|_{Y^0}.
\end{align}
\end{corollary}
\begin{proof}
By applying the $L^{p}$ interpolation inequality, the $L^{2}$ bilinear Strichartz estimate \eqref{equ:bilinear estimate,Strichartz,L2} in Lemma
\ref{lemma:bilinear,strchartz,L2}, and the Strichartz estimate \eqref{equ:strichartz estimate with noncertered frequency localization} with
 non-centered frequency localization in Lemma \ref{lemma:Strichartz estimate}, we obtain
\begin{align*}
&\n{\lrs{P_{Q} P_{N_1} u_{1}}  \lrs{P_{N_2} u_{2}}}_{L_{t,x}^{\frac{3}{2}}}\\
\lesssim & \n{\lrs{P_{Q} P_{N_1} u_{1}}  \lrs{P_{N_2} u_{2}}}_{L_{t,x}^{\frac{7}{5}}}^{1-\theta}
\n{\lrs{P_{Q} P_{N_1} u_{1}}  \lrs{P_{N_2} u_{2}}}_{L_{t,x}^{2}}^{\theta}\notag\\
\lesssim& \lrc{\n{P_{Q} P_{N_1} u_{1}}_{L_{t,x}^{\frac{14}{5}}}\n{P_{N_2} u_{2}}_{L_{t,x}^{\frac{14}{5}}}}^{1-\theta}
\lrc{M_{2}^{\frac{d-2}{2}}\lrs{\frac{M_{2}}{M_{3}}+\frac{1}{M_{2}}}^{\delta}\n{P_{Q} P_{N_1} u_{1}}_{Y^{0}}\n{P_{N_2} u_{2}}_{Y^{0}}}^{\theta}\notag\\
\lesssim& N_{2}^{\frac{d-4}{3}}\lrs{\frac{N_{2}}{N_{1}}+\frac{1}{N_{2}}}^{\delta \theta}\n{P_{Q} P_{N_1} u_{1}}_{Y^{0}}\n{P_{N_2} u_{2}}_{Y^{0}},\notag
\end{align*}
which yields the desired estimate,
\end{proof}

\noindent\textbf{High-low frequency bilinear estimates.}
With the above Strichartz estimates in place, we can establish the required $U$-$V$ bilinear estimates. We present the proof for the $\T^{d}$ case, as the treatment for
$\R^{d}$ proceeds in the analogous manner.
\begin{proposition}\label{lemma:bilinear estimate,d6}
Let $\wt{u}\in \lr{u,\ol{u}}$, $d\geq 6$, and $s\in\lrc{\frac{d-8}{2},\frac{d-4}{2}}$. We have the high frequency estimate
\begin{align}\label{equ:bilinear estimate,high}
\int_{0}^{T}\int_{\T^{d}}\wt{u}_{1}(t,x)\wt{u}_{2}(t,x)g(t,x)dxdt\lesssim \n{u_{1}}_{Y^{s}}\n{u_{2}}_{Y^{\frac{d-4}{2}}}
\n{g}_{Y^{-s}},
\end{align}
and the low frequency estimate
\begin{align}\label{equ:bilinear estimate,low}
\int_{0}^{T}\int_{\T^{d}}\wt{u}_{1}(t,x)(P_{\leq M_{0}}\wt{u}_{2})(t,x)g(t,x)dxdt
\lesssim& T^{\frac{1}{6}}M_{0}^{\frac{1}{3}} \n{u_{1}}_{Y^{s}}\n{P_{\leq M_{0}}u_{2}}_{Y^{\frac{d-4}{2}}}
\n{g}_{Y^{-s}},
\end{align}
for all $T\leq 1$ and all frequencies $M_{0}\geq 1$. Furthermore, we have
\begin{align} \label{equ:bilinear estimate,high,dual argument}
\bbn{\int_{t_{0}}^{t}e^{i(t-\tau)\Delta}(\wt{u}_{1}\wt{u}_{2})d\tau}_{X^{s}}\lesssim \n{u_{1}}_{X^{s}}
\n{u_{2}}_{X^{\frac{d-4}{2}}},
\end{align}
and
\begin{align} \label{equ:bilinear estimate,low,dual argument}
\bbn{\int_{t_{0}}^{t}e^{i(t-\tau)\Delta}(\wt{u}_{1}\wt{u}_{2})d\tau}_{X^{s}}
\lesssim& \n{u_{1}}_{X^{s}}
\lrs{T^{\frac{1}{6}}M_{0}^{\frac{1}{3}}\n{P_{\leq M_{0}}u_{2}}_{X^{\frac{d-4}{2}}}+\n{P_{>M_{0}}u_{2}}_{X^{\frac{d-4}{2}}}}.
\end{align}

\end{proposition}

\begin{proof}
It suffices to establish the high and low frequency estimates \eqref{equ:bilinear estimate,high} and \eqref{equ:bilinear estimate,low}, as
the dual estimates \eqref{equ:bilinear estimate,high,dual argument} and \eqref{equ:bilinear estimate,low,dual argument} are subsequently derived via Lemma \ref{lemma:U-V estimate,dual argument}.

Without loss of generality, we simplify the notation by setting $\wt{u}=u$.
Applying a Littlewood-Paley decomposition, we obtain
\begin{align*}
\int_{0}^{T}\int_{\T^{d}}u_{1}(t,x)(P_{\leq M_{0}}u_{2})(t,x)g(t,x)dxdt=\sum_{M_{1},M_{2},M_{3}}I_{M_{1},M_{2},M_{3}},
\end{align*}
where
\begin{align*}
I_{M_{1},M_{2},M_{3}}=\int_{0}^{T}\int_{\T^{d}}\lrs{u_{1,M_{1}}}\lrs{u_{2,M_{2}}}g_{M_{3}}dxdt
\end{align*}
with $u_{1,M_{1}}=P_{M_{1}}u_{1}$, $u_{2,M_{2}}=P_{M_{2}}u_{2}$ and $g_{M_{3}}=P_{M_{3}}g$. Due to the frequency constraint, we then divide the sum into three cases as follows.

Case A. $M_{1}\sim M_{3}\geq M_{2}$.

Case B. $M_{1}\sim M_{2}\geq M_{3}$.

Case C. $M_{2}\sim M_{3}\geq M_{1}$.

We only need to address Case A and Case B, as the proof for Case C follows in a similar manner.

\vspace{1em}
\noindent \textbf{Proof of the high frequency estimate $(\ref{equ:bilinear estimate,high})$.}

\noindent \textbf{Case A. $M_{1}\sim M_{3}\geq M_{2}$.}

Let $I_{A}$ denote the integral restricted to the Case A,
and let $\xi_{1}$, $\xi_{2}$, and $\xi_{3}$ denote the frequency variables corresponding to the
spatial variables $x_{1}$, $x_{2}$, and $x_{3}$, respectively. We proceed by decomposing the $M_{1}$ and $M_{3}$ dyadic spaces
into $M_{2}$-size cubes. Due to the frequency constraint $\xi_{1}+\xi_{2}+\xi_{3}=0$, for each
choice $Q$ of an $M_{2}$-size cube within the $\xi_{1}$ space, the variable $\xi_{3}$ is
constrained to at most $10^{d}$ of $M_{2}$-size cubes dividing the $M_{3}$ dyadic space.
To simplify the notation, we group these $10^{d}$ $M_{2}$-size cubes as a single cube $Q^{c}$ associated with each $Q$. Then we have
\begin{align}\label{equ:case A,high,Q}
I_{A}=\sum_{\substack{M_{1},M_{2},M_{3}\\M_{1}\sim M_{3}\geq  M_{2}}}\sum_{Q}\int_{0}^{T}\int_{\T^{d}}\lrs{P_{Q}u_{1,M_{1}}}\lrs{u_{2,M_{2}}}\lrs{P_{Q^{c}}g_{M_{3}}}dxdt.
\end{align}
By H\"{o}lder inequality, we get
\begin{align*}
I_{A}\lesssim& \sum_{\substack{M_{1},M_{2},M_{3}\\M_{1}\sim M_{3}\geq M_{2}}}\sum_{Q} \n{\lrs{P_{Q}u_{1,M_{1}}}\lrs{u_{2,M_{2}}}
\lrs{P_{Q^{c}}g_{M_{3}}}}_{L_{t,x}^{1}}\\
\lesssim& \sum_{\substack{M_{1},M_{2},M_{3}\\ M_{1} \sim M_{3}\geq  M_{2}}}\sum_{Q}\n{P_{Q}u_{1,M_{1}}}_{L_{t,x}^{3}}
\n{\lrs{u_{2,M_{2}}}\lrs{P_{Q^{c}}g_{M_{3}}}}_{L_{t,x}^{\frac{3}{2}}}.
\end{align*}
 By Strichartz estimates \eqref{equ:bilinear estimate,Strichartz,L32} and \eqref{equ:strichartz estimate with noncertered frequency localization}, and Cauchy-Schwarz inequality in $Q$, we obtain
\begin{align*}
I_{A}\lesssim& \sum_{\substack{M_{1},M_{2},M_{3}\\ M_{1}\sim M_{3} \geq M_{2}}}\sum_{Q}M_{2}^{\frac{d-4}{6}}M_{2}^{\frac{d-4}{3}}\lrs{\frac{M_{2}}{M_{3}}+\frac{1}{M_{2}}}^{\delta}\n{P_{Q}u_{1,M_{1}}}_{Y^{0}}
\n{u_{2,M_{2}}}_{Y^{0}}
\n{P_{Q^{c}}g_{M_{3}}}_{Y^{0}}\\
\lesssim& \sum_{\substack{M_{1},M_{2},M_{3}\\ M_{1}\sim M_{3} \geq M_{2}}}M_{2}^{\frac{d-4}{2}}\lrs{\frac{M_{2}}{M_{3}}+\frac{1}{M_{2}}}^{\delta}\n{u_{1,M_{1}}}_{Y^{0}}
\n{u_{2,M_{2}}}_{Y^{0}}
\n{g_{M_{3}}}_{Y^{0}}\\
\lesssim& \sum_{\substack{M_{1},M_{2},M_{3}\\ M_{1}\sim M_{3} \geq M_{2}}}\lrs{\frac{M_{2}}{M_{3}}+\frac{1}{M_{2}}}^{\delta}\n{u_{1,M_{1}}}_{Y^{s}}
\n{u_{2,M_{2}}}_{Y^{\frac{d-4}{2}}}
\n{g_{M_{3}}}_{Y^{-s}}.
\end{align*}
Supping out $M_{2}$ and applying Cauchy-Schwarz inequality, we get
\begin{align}\label{equ:IA,bilinear,U-V}
I_{A}\lesssim &\n{u_{2}}_{Y^{\frac{d-4}{2}}}\sum_{\substack{M_{1},M_{2},M_{3}\\ M_{1}\sim M_{3} \geq M_{2}}}\lrs{\frac{M_{2}}{M_{3}}+\frac{1}{M_{2}}}^{\delta}\n{u_{1,M_{1}}}_{Y^{s}}
\n{g_{M_{3}}}_{Y^{-s}}\\
\lesssim& \n{u_{2}}_{Y^{\frac{d-4}{2}}}\sum_{\substack{M_{1},M_{3}\\ M_{1}\sim M_{3}}}\n{u_{1,M_{1}}}_{Y^{s}}
\n{g_{M_{3}}}_{Y^{-s}}\notag\\
\lesssim& \n{u_{1}}_{Y^{s}} \n{u_{2}}_{Y^{\frac{d-4}{2}}}\n{g}_{Y^{-s}},\notag
\end{align}
which completes the proof of the high frequency estimate for Case A.

\vspace{1em}

\noindent \textbf{Case B. $M_{1}\sim M_{2}\geq M_{3}$.}

Let $I_{B}$ denote the integral restricted to the Case B. Similar to the decomposition for
 Case A presented in (\ref{equ:case A,low,Q}), we decompose the $M_{1}$ and $M_{2}$ dyadic spaces
 into cubes of size $M_{3}$. This yields
\begin{align*}
I_{B}=\sum_{\substack{M_{1},M_{2},M_{3}\\
M_{1}\sim M_{2}\geq M_{3}}}
\sum_{Q}\int_{0}^{T}\int_{\T^{d}}\lrs{P_{Q}u_{1,M_{1}}}\lrs{P_{Q^{c}}u_{2,M_{2}}}g_{M_{3}}dxdt,
\end{align*}
where $Q$ and $Q^{c}$ are of size $M_{3}$.
By H\"{o}lder inequality, we obtain
\begin{align*}
I_{B}\lesssim& \sum_{\substack{M_{1},M_{2},M_{3}\\
M_{1}\sim M_{2}\geq M_{3}}}\sum_{Q} \n{\lrs{P_{Q}u_{1,M_{1}}}\lrs{P_{Q^{c}}u_{2,M_{2}}}
g_{M_{3}}}_{L_{t,x}^{1}}\\
\lesssim& \sum_{\substack{M_{1},M_{2},M_{3}\\ M_{1}\sim M_{2}\geq M_{3}}}\sum_{Q}\n{P_{Q}u_{1,M_{1}}}_{L_{t,x}^{3}}
\n{\lrs{P_{Q^{c}}u_{2,M_{2}}}\lrs{g_{M_{3}}}}_{L_{t,x}^{\frac{3}{2}}}.
\end{align*}
 By Strichartz estimates \eqref{equ:bilinear estimate,Strichartz,L32} and \eqref{equ:strichartz estimate with noncertered frequency localization}, and Cauchy-Schwarz inequality in $Q$, we derive
\begin{align*}
I_{B}\lesssim& \sum_{\substack{M_{1},M_{2},M_{3}\\ M_{1}\sim M_{2}\geq M_{3}}}\sum_{Q}M_{3}^{\frac{d-4}{6}}M_{3}^{\frac{d-4}{3}}\lrs{\frac{M_{3}}{M_{2}}+\frac{1}{M_{3}}}^{\delta}\n{P_{Q}u_{1,M_{1}}}_{Y^{0}}
\n{P_{Q^{c}}u_{2,M_{2}}}_{Y^{0}}
\n{g_{M_{3}}}_{Y^{0}}\\
\lesssim& \sum_{\substack{M_{1},M_{2},M_{3}\\ M_{1}\sim M_{2}\geq M_{3}}}M_{3}^{\frac{d-4}{2}}\lrs{\frac{M_{3}}{M_{2}}+\frac{1}{M_{3}}}^{\delta}\n{u_{1,M_{1}}}_{Y^{0}}
\n{u_{2,M_{2}}}_{Y^{0}}
\n{g_{M_{3}}}_{Y^{0}}\\
\lesssim& \sum_{\substack{M_{1},M_{2},M_{3}\\ M_{1}\sim M_{2}\geq M_{3}}}M_{1}^{-s}M_{2}^{-\frac{d-4}{2}}M_{3}^{\frac{d-4}{2}+s}\lrs{\frac{M_{3}}{M_{2}}+\frac{1}{M_{3}}}^{\delta}\n{u_{1,M_{1}}}_{Y^{s}}
\n{u_{2,M_{2}}}_{Y^{\frac{d-4}{2}}}
\n{g_{M_{3}}}_{Y^{-s}}\\
\lesssim& \sum_{\substack{M_{1},M_{2},M_{3}\\ M_{1}\sim M_{2}\geq M_{3}}}\lrs{\frac{M_{3}}{M_{2}}+\frac{1}{M_{3}}}^{\delta}\n{u_{1,M_{1}}}_{Y^{s}}
\n{u_{2,M_{2}}}_{Y^{\frac{d-4}{2}}}
\n{g_{M_{3}}}_{Y^{-s}}.
\end{align*}
In a similar way as the estimate \eqref{equ:IA,bilinear,U-V} for $I_{A}$,
 we get
\begin{align*}
I_{B}
\lesssim& \n{g}_{Y^{-s}}\sum_{\substack{M_{1},M_{2}\\ M_{1}\sim M_{2}}}\n{u_{1,M_{1}}}_{Y^{s}}
\n{u_{2,M_{2}}}_{Y^{\frac{d-4}{2}}}\\
\lesssim& \n{u_{1}}_{Y^{s}} \n{u_{2}}_{Y^{\frac{d-4}{2}}}\n{g}_{Y^{-s}},
\end{align*}
which completes the proof of the high frequency estimate for Case B.

\vspace{1em}

\noindent\textbf{Proof of the low frequency estimate $(\ref{equ:bilinear estimate,low})$.}

\noindent \textbf{Case A. $M_{1}\sim M_{3}\geq M_{2}$.}

Analogous to the decomposition of Case A for high frequency estimate presented in (\ref{equ:case A,high,Q}), we decompose the $M_{1}$ and $M_{3}$ dyadic spaces
 into cubes of size $M_{2}$, which yields
\begin{align}\label{equ:case A,low,Q}
I_{A}=\sum_{\substack{M_{1},M_{2},M_{3}\\M_{1}\sim M_{3}\geq  M_{2}}}\sum_{Q}\int_{0}^{T}\int_{\T^{d}}\lrs{P_{Q}u_{1,M_{1}}}\lrs{P_{\leq M_{0}} u_{2,M_{2}}}\lrs{P_{Q^{c}}g_{M_{3}}}dxdt.
\end{align}
By H\"{o}lder inequality, we deduce
\begin{align*}
I_{A}\lesssim& \sum_{\substack{M_{1},M_{2},M_{3}\\M_{1}\sim M_{3}\geq M_{2}}}\sum_{Q} \n{\lrs{P_{Q}u_{1,M_{1}}}\lrs{P_{\leq M_{0}}u_{2,M_{2}}}
\lrs{P_{Q^{c}}g_{M_{3}}}}_{L_{t,x}^{1}}\\
\lesssim& \sum_{\substack{M_{1},M_{2},M_{3}\\ M_{1} \sim M_{3}\geq  M_{2}}}\sum_{Q}\n{P_{Q}u_{1,M_{1}}}_{L_{t,x}^{3}}
\n{\lrs{P_{\leq M_{0}}u_{2,M_{2}}}\lrs{P_{Q^{c}}g_{M_{3}}}}_{L_{t,x}^{\frac{3}{2}}}.
\end{align*}
Applying H\"{o}lder inequality, Bernstein inequality, Strichartz estimate \eqref{equ:strichartz estimate with noncertered frequency localization}, and  estimate \eqref{equ:u-v and sobolev}, we have
\begin{align}\label{equ:L23,low,T}
&\n{\lrs{P_{\leq M_{0}}u_{2,M_{2}}}\lrs{P_{Q^{c}}g_{M_{3}}}}_{L_{t,x}^{\frac{3}{2}}}\\
\lesssim&  \n{P_{\leq M_{0}}u_{2,M_{2}}}_{L_{t,x}^{3}}\n{P_{Q^{c}}g_{M_{3}}}_{L_{t,x}^{3}}\notag\\
\lesssim& T^{\frac{1}{3}}M_{0}^{\frac{2}{3}}M_{2}^{\frac{d-4}{6}}\n{P_{\leq M_{0}}u_{2,M_{2}}}_{L_{t}^{\infty}L_{x}^{2}} M_{2}^{\frac{d-4}{6}}
\n{P_{Q^{c}}g_{M_{3}}}_{Y^{0}}\notag\\
\lesssim& T^{\frac{1}{3}}M_{0}^{\frac{2}{3}}M_{2}^{\frac{d-4}{3}}\n{P_{\leq M_{0}}u_{2,M_{2}}}_{Y^{0}}
\n{P_{Q^{c}}g_{M_{3}}}_{Y^{0}}.\notag
\end{align}
One half of the above term is estimated using estimate \eqref{equ:L23,low,T}, and the complementary half is estimated using the $L^{\frac{3}{2}}$ Strichartz estimate \eqref{equ:bilinear estimate,Strichartz,L32}. We then obtain
\begin{align}\label{equ:low,IA}
I_{A}\lesssim& T^{\frac{1}{6}}M_{0}^{\frac{1}{3}}\sum_{\substack{M_{1},M_{2},M_{3}\\ M_{1} \sim M_{3}\geq  M_{2}}}\sum_{Q}
M_{2}^{\frac{d-4}{3}}\lrs{\frac{M_{2}}{M_{3}}+\frac{1}{M_{2}}}^{\frac{\delta}{2}}
\n{P_{Q}u_{1,M_{1}}}_{L_{t,x}^{3}}
\n{P_{\leq M_{0}}u_{2,M_{2}}}_{Y^{0}}\n{P_{Q^{c}}g_{M_{3}}}_{Y^{0}}.
\end{align}
Applying Strichartz estimate \eqref{equ:strichartz estimate with noncertered frequency localization} and Cauchy-Schwarz inequality in
$Q$, we derive
\begin{align*}
I_{A}\lesssim& T^{\frac{1}{6}}M_{0}^{\frac{1}{3}}\sum_{\substack{M_{1},M_{2},M_{3}\\ M_{1}\sim M_{3} \geq M_{2}}}\sum_{Q}M_{2}^{\frac{d-4}{2}}\lrs{\frac{M_{2}}{M_{3}}+\frac{1}{M_{2}}}^{\frac{\delta}{2}}\n{P_{Q}u_{1,M_{1}}}_{Y^{0}}
\n{P_{\leq M_{0}}u_{2,M_{2}}}_{Y^{0}}
\n{P_{Q^{c}}g_{M_{3}}}_{Y^{0}}\\
\lesssim& T^{\frac{1}{6}}M_{0}^{\frac{1}{3}}\sum_{\substack{M_{1},M_{2},M_{3}\\ M_{1}\sim M_{3} \geq M_{2}}}M_{2}^{\frac{d-4}{2}}\lrs{\frac{M_{2}}{M_{3}}+\frac{1}{M_{2}}}^{\frac{\delta}{2}}\n{u_{1,M_{1}}}_{Y^{0}}
\n{P_{\leq M_{0}}u_{2,M_{2}}}_{Y^{0}}
\n{g_{M_{3}}}_{Y^{0}}\\
\lesssim& T^{\frac{1}{6}}M_{0}^{\frac{1}{3}}\sum_{\substack{M_{1},M_{2},M_{3}\\ M_{1}\sim M_{3} \geq M_{2}}}\lrs{\frac{M_{2}}{M_{3}}+\frac{1}{M_{2}}}^{\frac{\delta}{2}}\n{u_{1,M_{1}}}_{Y^{s}}
\n{P_{\leq M_{0}}u_{2,M_{2}}}_{Y^{\frac{d-4}{2}}}
\n{g_{M_{3}}}_{Y^{-s}}.
\end{align*}
Supping out $M_{2}$ and applying Cauchy-Schwarz inequality, we get
\begin{align}
I_{A}\lesssim &T^{\frac{1}{6}}M_{0}^{\frac{1}{3}}\n{P_{\leq M_{0}}u_{2}}_{Y^{\frac{d-4}{2}}}\sum_{\substack{M_{1},M_{2},M_{3}\\ M_{1}\sim M_{3} \geq M_{2}}}\lrs{\frac{M_{2}}{M_{3}}+\frac{1}{M_{2}}}^{\frac{\delta}{2}}\n{u_{1,M_{1}}}_{Y^{s}}
\n{g_{M_{3}}}_{Y^{-s}}\label{equ:IA,low,U-V}\\
\lesssim& T^{\frac{1}{6}}M_{0}^{\frac{1}{3}}\n{P_{\leq M_{0}}u_{2}}_{Y^{\frac{d-4}{2}}}\sum_{\substack{M_{1},M_{3}\\ M_{1}\sim M_{3}}}\n{u_{1,M_{1}}}_{Y^{s}}
\n{g_{M_{3}}}_{Y^{-s}}\notag\\
\lesssim& T^{\frac{1}{6}}M_{0}^{\frac{1}{3}}\n{u_{1}}_{Y^{s}} \n{P_{\leq M_{0}}u_{2}}_{Y^{\frac{d-4}{2}}}\n{g}_{Y^{-s}},\notag
\end{align}
which completes the proof of the low frequency estimate for Case A.

\vspace{1em}

\noindent \textbf{Case B. $M_{1}\sim M_{2}\geq M_{3}$.}

Similar to the decomposition for
 Case A presented in (\ref{equ:case A,low,Q}), we decompose the $M_{1}$ and $M_{2}$ dyadic spaces
 into cubes of size $M_{3}$. Then we have
\begin{align*}
I_{B}=\sum_{\substack{M_{1},M_{2},M_{3}\\
M_{1}\sim M_{2}\geq M_{3}}}
\sum_{Q}\int_{0}^{T}\int_{\T^{d}}\lrs{P_{Q}u_{1,M_{1}}}\lrs{P_{Q^{c}}P_{\leq M_{0}}u_{2,M_{2}}}g_{M_{3}}dxdt.
\end{align*}
Adopting the same treatment as in \eqref{equ:L23,low,T}--\eqref{equ:low,IA}, we get
\begin{align*}
I_{B}\lesssim& \sum_{\substack{M_{1},M_{2},M_{3}\\
M_{1}\sim M_{2}\geq M_{3}}}\sum_{Q} \n{\lrs{P_{Q}u_{1,M_{1}}}\lrs{P_{Q^{c}}P_{\leq M_{0}}u_{2,M_{2}}}
g_{M_{3}}}_{L_{t,x}^{1}}\\
\lesssim& \sum_{\substack{M_{1},M_{2},M_{3}\\ M_{1}\sim M_{2}\geq M_{3}}}\sum_{Q}\n{P_{Q}u_{1,M_{1}}}_{L_{t,x}^{3}}
\n{\lrs{P_{Q^{c}}P_{\leq M_{0}}u_{2,M_{2}}}\lrs{g_{M_{3}}}}_{L_{t,x}^{\frac{3}{2}}}\\
\lesssim&T^{\frac{1}{6}}M_{0}^{\frac{1}{3}}\sum_{\substack{M_{1},M_{2},M_{3}\\ M_{1} \sim M_{2}\geq  M_{3}}}\sum_{Q}
M_{3}^{\frac{d-4}{3}}\lrs{\frac{M_{3}}{M_{2}}+\frac{1}{M_{3}}}^{\frac{\delta}{2}}
\n{P_{Q}u_{1,M_{1}}}_{L_{t,x}^{3}}
\n{P_{Q^{c}}P_{\leq M_{0}}u_{2,M_{2}}}_{Y^{0}}\n{g_{M_{3}}}_{Y^{0}}.
\end{align*}
 By Strichartz estimate \eqref{equ:strichartz estimate with noncertered frequency localization} and Cauchy-Schwarz inequality in $Q$, we derive
\begin{align*}
I_{B}\lesssim& T^{\frac{1}{6}}M_{0}^{\frac{1}{3}}\sum_{\substack{M_{1},M_{2},M_{3}\\ M_{1}\sim M_{2}\geq M_{3}}}\sum_{Q}M_{3}^{\frac{d-4}{2}}\lrs{\frac{M_{3}}{M_{2}}+\frac{1}{M_{3}}}^{\frac{\delta}{2}}\n{P_{Q}u_{1,M_{1}}}_{Y^{0}}
\n{P_{Q^{c}}P_{\leq M_{0}}u_{2,M_{2}}}_{Y^{0}}
\n{g_{M_{3}}}_{Y^{0}}\\
\lesssim& T^{\frac{1}{6}}M_{0}^{\frac{1}{3}}\sum_{\substack{M_{1},M_{2},M_{3}\\ M_{1}\sim M_{2}\geq M_{3}}}M_{3}^{\frac{d-4}{2}}\lrs{\frac{M_{3}}{M_{2}}+\frac{1}{M_{3}}}^{\frac{\delta}{2}}\n{u_{1,M_{1}}}_{Y^{0}}
\n{P_{\leq M_{0}}u_{2,M_{2}}}_{Y^{0}}
\n{g_{M_{3}}}_{Y^{0}}\\
\lesssim& T^{\frac{1}{6}}M_{0}^{\frac{1}{3}}\sum_{\substack{M_{1},M_{2},M_{3}\\ M_{1}\sim M_{2}\geq M_{3}}}\lrs{\frac{M_{3}}{M_{2}}+\frac{1}{M_{3}}}^{\frac{\delta}{2}}\n{u_{1,M_{1}}}_{Y^{s}}
\n{P_{\leq M_{0}}u_{2,M_{2}}}_{Y^{\frac{d-4}{2}}}
\n{g_{M_{3}}}_{Y^{-s}}.
\end{align*}
In the same way as \eqref{equ:IA,low,U-V} for $I_{A}$, we get
\begin{align*}
I_{B}\lesssim &T^{\frac{1}{6}}M_{0}^{\frac{1}{3}}\n{u_{1}}_{Y^{s}} \n{P_{\leq M_{0}}u_{2}}_{Y^{\frac{d-4}{2}}}\n{g}_{Y^{-s}},
\end{align*}
which completes the proof of the low frequency estimate for Case B.
\end{proof}

\section{Proof of the Main Theorem}\label{sec:Proof of the Main Theorem}
\begin{proof}[\textbf{Proof of Theorem $\ref{thm:uniqueness for nls}$}]
Let $\phi_{1}$ and $\phi_{2}$ be two $C([0,T];H^{s_{c}})$
solutions to the quadratic NLS \eqref{equ:NLS} with the same initial datum $\phi_{0}$. We set
\begin{align}
\Phi^{(k)}(t)=\int \phi^{\otimes k}d\nu_{t}(\phi),
\end{align}
where $\nu_{t}(\phi)=\delta_{\phi_{1}(t)}(\phi)-\delta_{\phi_{2}(t)}(\phi)$. Then by Proposition \ref{lemma:upper echelon form, time integration domain} and the equivalence of compatible time integration domain as shown in \eqref{equ:compatible,Tc,Td}, we rewrite
\begin{align*}
\Phi^{(1)}(t_{1})=&\sum_{\mu\in \mathrm{C}^{(k)}_{\mathrm{up}} }\int_{T_{D}(\mu)}J_{\mu}^{(k+1)}(\Phi^{(k+1)})(t_{1},\underline{t}_{k+1})
d\underline{t}_{k+1}\\
=&\sum_{\mu\in \mathrm{C}^{(k)}_{\mathrm{up}}}\int_{T_{C}(\mu)}J_{\mu}^{(k+1)}(\Phi^{(k+1)})(t_{1},\underline{t}_{k+1})
d\underline{t}_{k+1}
\end{align*}
where there are at most $4^{k}$ terms inside the the class $\mathrm{C}^{(k)}_{\mathrm{up}}$ of the upper echelon form. By iterative estimate \eqref{equ:estimate for the compatible time integration domain} in Proposition \ref{lemma:estimate for the compatible time integration domain}, we have
\begin{align*}
&\n{\lra{\nabla_{x_{1}}}^{\frac{d-8}{2}}\Phi^{(1)}(t_{1})}_{L_{T}^{\infty}L_{x_{1}}^{2}}\\
\leq& \sum_{\mu\in \mathrm{C}^{(k)}_{\mathrm{up}}}
\bbn{\lra{\nabla_{x_{1}}}^{\frac{d-8}{2}}\int_{T_{C}(\mu)}J_{\mu}^{(k+1)}
(\Phi^{(k+1)})
(t_{1},\underline{t}_{k+1})d\underline{t}_{k+1}}
_{L_{T}^{\wq}L_{x_{1}}^{2}}\\
\leq& \sum_{\mu\in \mathrm{C}^{(k)}_{\mathrm{up}}}C^{k}\int_{0}^{T}\int \n{\phi}_{H^{\frac{d-4}{2}}}^{\frac{k}{2}+1}\lrs{T^{\frac{1}{6}}M_{0}^{\frac{1}{3}}\n{P_{\leq M_{0}}\phi}_{H^{\frac{d-4}{2}}}+\n{P_{>M_{0}}\phi}_{H^{\frac{d-4}{2}}}}^{\frac{k}{2}}
d|\nu_{t_{k+1}}|(\phi)dt_{k+1}
\end{align*}
Using the UTFL property \eqref{equ:UTFL for NLS} in Lemma \ref{lemma:UTFL}, we derive
\begin{align*}
\n{\lra{\nabla_{x_{1}}}^{\frac{d-8}{2}}\Phi^{(1)}(t_{1})}_{L_{T}^{\infty}L_{x_{1}}^{2}}\leq &
2T(4C)^{k}(C_{0})^{\frac{k}{2}+1}\lrs{T^{\frac{1}{6}}M_{0}^{\frac{1}{3}}C_{0}+\ve}^{\frac{k}{2}}\\
=&2TC_{0}\lrs{16C^{2}C_{0}^{2}T^{\frac{1}{6}}M_{0}^{\frac{1}{3}}+16C^{2}C_{0}\ve}^{\frac{k}{2}}
\end{align*}
where $C_{0}=\max\lr{\n{\phi_{1}}_{L_{T_{0}}^{\infty}H^{\frac{d-4}{2}}},\n{\phi_{2}}_{L_{T_{0}}^{\infty}H^{\frac{d-4}{2}}}}$.
 Selecting $\ve$ small enough such that $16 C^{2}C_{0}\ve< \frac{1}{4}$ and then selecting $T$ small enough such that $16C^{2}C_{0}^{2}T^{\frac{1}{6}}M_{0}^{\frac{1}{3}}< \frac{1}{4}$, we thus have for $T$ small enough,
\begin{align*}
\n{\Phi^{(1)}(t_{1})}_{L_{T}^{\infty}L_{x,\xi}^{2}}\leq 2TC_{0}\lrs{\frac{1}{2}}^{k}\to 0, \quad \text{as $k\to \infty$.}
\end{align*}
Therefore, we have concluded that $\phi(t)=0$ for $t\in [0,T]$. By a bootstrap argument, we can then fill the whole $[0,T_{0}]$ interval.
\end{proof}

\noindent \textbf{Acknowledgements} S. Shen would like to thank Professors Xuwen Chen, Justin Holmer, Changxing Miao, and Zhifei Zhang for encouraging and helpful discussions.
S. Shen was supported in part by the National Key R\&D Program of China under Grant 2024YFA1015500, NSF of China under Grant 12501322, and Anhui Provincial NSF 2508085QA001.

\noindent\textbf{Data Availability Statement}
Data sharing is not applicable to this article as no datasets were generated or analysed during the current study.

\noindent\textbf{Conflict of Interest}
The authors declare that they have no conflict of interest.

\appendix
\section{Uniform in Time Frequency Localization for Quadratic NLS}\label{sec:Uniform in Time Frequency Localization for Quadratic NLS}
The uniform-in-time frequency localization property (UTFL) originates from the uniqueness analysis of the GP hierarchy in \cite{CH19}. In the present quadratic NLS setting, we give a proof of this property, adapting the arguments from \cite{CH22unc,CSZ22}.
\begin{lemma}\label{lemma:UTFL}
Let $u$ be a $C([0,T_{0}];H^{s_{c}})$ solution to the quadratic NLS \eqref{equ:NLS}. Then $u$ satisfies the uniform-in-time frequency localization property. More precisely, for every $\ve>0$ there exists $M(\ve)$ such that
\begin{equation} \label{equ:UTFL for NLS}
\n{\lra{\nabla}^{s_{c}}P_{> M(\ve)}u}_{L_{T_{0}}^{\infty}L_{x}^{2}}\leq \ve.
\end{equation}
\end{lemma}
\begin{proof}
We compute
\begin{align*}
\babs{\pa_{t}\n{\lra{\nabla}^{s_{c}} P_{\leq M}u}_{L_{x}^{2}}^{2}}=&2\bbabs{\operatorname{Im} \int P_{\leq M}\lra{\nabla}^{s_{c}} u\cdot P_{\leq M}\lra{\nabla}^{s_{c}}B(u,u)dx}\\
\leq& 2\n{P_{\leq M}\lra{\nabla}^{s_{c}} u}_{L^{2}}\n{P_{\leq M}\lra{\nabla}^{s_{c}}B(u,u)}_{L^{2}}.
\end{align*}
Noting that $\n{P_{\leq M}\lra{\nabla}^{s}f}_{L^{2}}\lesssim M^{s}\n{P_{\leq M}f}_{L^{2}}$, then by Bernstein inequality and Sobolev inequality, we have
\begin{align*}
\babs{\pa_{t}\n{\lra{\nabla}^{s_{c}} P_{\leq M}u}_{L_{x}^{2}}^{2}}\lesssim & 2M^{2}\n{ P_{\leq M}\lra{\nabla}^{s_{c}}u}_{L^{2}}\n{P_{\leq M}\lra{\nabla}^{s_{c}-2}B(u,u)}_{L^{2}}\\
\lesssim& 2M^{2}\n{u}_{H^{s_{c}}}^{2}.
\end{align*}
Thus, there is $\delta'>0$ such that for any $t_{0}\in[0,T_{0}]$ and  $t\in(t_{0}-\delta',t_{0}+\delta')\cap[0,T_{0}]$, we have
\begin{align*}
\bbabs{\n{\lra{\nabla}^{s_{c}} P_{\leq M} u(t)}_{L_{x}^{2}}^{2}-\n{\lra{\nabla}^{s_{c}} P_{\leq M}u(t_{0})}_{L_{x}^{2}}^{2}}\leq \frac{1}{16}\ve^{2}.
\end{align*}
Since $u$ is continuous in $H^{s_{c}}$, for each $t_{0}$ we can find $\delta''>0$ such that
\begin{align*}
\bbabs{\n{\lra{\nabla}^{s_{c}}  u(t)}_{L_{x}^{2}}^{2}-\n{\lra{\nabla}^{s_{c}} u(t_{0})}_{L_{x}^{2}}^{2}}\leq \frac{1}{16}\ve^{2}
\end{align*}
for all $t\in(t_{0}-\delta'',t_{0}+\delta'')\cap[0,T_{0}]$. With $\delta=\min\{\delta',\delta''\}$, we get for $t\in(t_{0}-\delta,t_{0}+\delta)\cap[0,T_{0}]$,
\begin{align*}
\bbabs{\n{\lra{\nabla}^{s_{c}}P_{>M}u(t)}_{L_{x}^{2}}^{2}-\n{\lra{\nabla}^{s_{c}}P_{>M}u(t_{0})}_{L_{x}^{2}}^{2}}\leq\frac{1}{4}\ve^{2}.
\end{align*}

On the other hand, for each fixed $t_{0}\in[0,T_{0}]$, there exists $M_{t_{0}}$ such that
\begin{align*}
\n{\lra{\nabla}^{s_{c}}P_{>M_{t_{0}}}u(t_{0})}_{L_{x}^{2}}\leq \frac{1}{2}\ve.
\end{align*}
Therefore, there is a $\delta_{t_{0}}>0$ such that for $t\in (t_{0}-\delta_{t_{0}},t_{0}+\delta_{t_{0}})\cap [0,T_{0}]$,
\begin{align*}
\n{\lra{\nabla}^{s_{c}}P_{>M_{t_{0}}}u(t)}_{L_{x}^{2}}\leq \ve.
\end{align*}
The collection of intervals $\{(t-\delta_{t},t+\delta_{t})\cap [0,T_{0}]\}_{t\in[0,T_{0}]}$ forms an open cover of the compact set $[0,T_{0}]$. Hence, we may extract a finite subcover, which we denote by
\begin{align*}
(t_{1}-\delta_{t_{1}},t_{1}+\delta_{t_{1}})\cap [0,T_{0}],\ldots,(t_{J}-\delta_{t_{J}},t_{J}+\delta_{t_{J}})\cap [0,T_{0}].
\end{align*}
Setting
$$M(\ve)=\max\{M_{t_{1}},\ldots,M_{t_{J}}\},$$
we obtain the desired estimate \eqref{equ:UTFL for NLS}.
\end{proof}

\bibliographystyle{abbrv}
%\bibliographystyle{plain}
%\nocite{*}
\bibliography{references}

\end{document}